\documentclass[12pt]{article}
\usepackage[affil-it]{authblk}
\usepackage{etex}
\usepackage[utf8]{inputenc}
\usepackage[T1]{fontenc}
\usepackage{textcomp}
\usepackage[a4paper,tmargin=3cm,bmargin=3cm,rmargin=2.2cm,lmargin=2.2cm]{geometry}
\usepackage{lmodern}
\usepackage[normalem]{ulem}

\usepackage{amsfonts,amssymb,amsmath,amscd,mathtools,cite,slashed}
\usepackage[amsthm,thmmarks,amsmath]{ntheorem}
\usepackage{graphicx}
\usepackage[all]{xy}
\usepackage{lmodern}
\usepackage{tikz}
\usetikzlibrary{shapes.misc,arrows,decorations.markings}

\usepackage{color}
\usepackage{psfrag}
\usepackage{epsfig}
\usepackage{pdflscape}
\usepackage{hyperref}
\usepackage{verbatim} 
\usepackage{graphicx}

\usepackage{xcolor}
\usepackage{framed}
\definecolor{shadecolor}{gray}{0.8}
\definecolor{lgray}{gray}{0.5}

\DeclareMathOperator{\Ker}{Ker}           
\DeclareMathOperator{\Res}{Res}         
\DeclareMathOperator{\Tr}{Tr}             
\DeclareMathOperator{\OO}{\mathcal{O}}             
\DeclareMathOperator{\Dom}{Dom}			
\DeclareMathOperator{\Sd}{Sd}			

\newcommand{\A}{\mathcal{A}}              
\renewcommand{\a}{\alpha}                    
\newcommand{\B}{\mathcal{B}}              
\newcommand{\C}{\mathbb{C}}              
\newcommand{\dt}{\partial}                    
\newcommand{\DD}{\mathcal{D}}           
\renewcommand{\H}{\mathcal{H}}          
\newcommand{\half}{{\mathchoice{\thalf}{\thalf}{\shalf}{\shalf}}}   
\newcommand{\Lc}{\mathcal{L}}          
\newcommand{\N}{\mathbb{N}}            
\newcommand{\R}{\mathbb{R}}             
\newcommand{\set}[1]{\{\,#1\,\}}              
\newcommand{\shalf}{{\scriptstyle\frac{1}{2}}} 
\newcommand{\thalf}{\tfrac{1}{2}}            
\newcommand{\x}{\times}                      
\newcommand{\wt}{\widetilde}                 
\newcommand{\Z}{\mathbb{Z}}                 
\newcommand{\Uq}{{\mathcal U}_q(\mathfrak{su}(2))}  
\newcommand{\vc}{\vcentcolon =}             
\newcommand{\cv}{= \vcentcolon}             

\def\<#1,#2>{\langle#1\,,\,#2\rangle}      
\newcommand{\norm}[1]{\left\lVert#1\right\rVert}    
\newbox\ncintdbox \newbox\ncinttbox
\setbox0=\hbox{$-$} \setbox2=\hbox{$\displaystyle\int$}
\setbox\ncintdbox=\hbox{\rlap{\hbox
    to \wd2{\kern-.1em\box2\relax\hfil}}\box0\kern.1em}
\setbox0=\hbox{$\vcenter{\hrule width 4pt}$}
\setbox2=\hbox{$\textstyle\int$}
\setbox\ncinttbox=\hbox{\rlap{\hbox
    to \wd2{\kern-.14em\box2\relax\hfil}}\box0\kern.1em}
\newcommand{\ncint}{\mathop{\mathchoice{\copy\ncintdbox}
    {\copy\ncinttbox}{\copy\ncinttbox}
    {\copy\ncinttbox}}\nolimits}
\newcommand{\oh}{{\tfrac{1}{2}}}          
\newcommand{\ket}[1]{| #1 \rangle}                  
\newcommand{\bra}[1]{\langle #1|}                  
\newcommand{\lpo}{l \! + \! 1}
\newcommand{\lmo}{l \! - \! 1}
\newcommand{\mpo}{m \! + \! 1}
\newcommand{\mmo}{m \! - \! 1}
\newcommand{\PDO}{\varPsi{}DO}				
\newcommand{\ahd}{(\mathcal{A,H,D})}		
\renewcommand{\<}{\left\langle}		
\renewcommand{\>}{\right\rangle}		
\renewcommand{\Re}{\mathfrak{R}}			

\newcommand{\ketQS}{\ket{l,m}_{\pm}}
\newcommand{\braQS}{{}_{\pm} \bra{l,m}}

\renewcommand{\bar}{\overline}		

\newcommand{\UqSU}{\mathcal{U}_q(\mathfrak{su}(2))} 		
\newcommand{\Sq}{S_q^2}			
\DeclareMathOperator{\OP}{OP}		
\DeclareMathOperator{\op}{op}		
\newcommand{\zD}{\zeta_{\DD}}		

\DeclareMathOperator{\Fp}{Fp}			
\newcommand{\Ag}{\mathbb{A}}			
\newcommand{\DA}{\DD_{\Ag}}			

\newcommand{\sump}[2]{\underset{#1}{\overset{#2}{\sum \, \!'}}}

\def\oq{\frac{1}{q}}

\newcounter{mnotecount}[section]
\renewcommand{\themnotecount}{\thesection.\arabic{mnotecount}}
\newcommand{\mnote}[1]
{\protect{\stepcounter{mnotecount}}$^{\mbox{\footnotesize
$
\bullet$\themnotecount}}$ \marginpar{
\raggedright\tiny\em
$\!\!\!\!\!\!\,\bullet$\themnotecount: #1} }

\newtheorem{assumption}{Assumption}[section]
\newtheorem{theorem}[assumption]{Theorem}

\newtheorem{Cor}[assumption]{Corollary}

\newtheorem{Lm}[assumption]{Lemma}
\newtheorem{lemma}[assumption]{Lemma}

\newtheorem{Def}[assumption]{Definition}
\newtheorem{definition}[assumption]{Definition}
\newtheorem{prop}[assumption]{Proposition}
\newtheorem{Prop}[assumption]{Proposition}
\newtheorem{rem}[assumption]{Remark}
\newtheorem{remark}[assumption]{Remark}

\newtheorem{appendixlemma}{Lemma}
\begin{document}


\title{Heat trace and spectral action \\ on the standard Podle\'s sphere}
\author{Micha\l\  Eckstein${}^1$}
\author{Bruno Iochum${}^2$}
\author{Andrzej Sitarz \thanks{Partially supported by NCN grant 2011/01/B/ST1/06474.} ${}^{,3,4}$}

\affil{\small ${}^1$ Faculty of Mathematics and Computer Science, Jagiellonian University,\\
{\L}{}ojasiewicza 6, 30-348 Krak\'ow, Poland,\\
michal.eckstein@uj.edu.pl}
\affil{\small ${}^2$Aix Marseille Université, CNRS et Université de Toulon. \\
CPT, UMR 7332, 13288 Marseille, France. Laboratoire affili\'e \`a la Frumam FR7332, \\
iochum@cpt.univ-mrs.fr}
\affil{\small ${}^3$ Institute of Physics, Jagiellonian University,\\
Reymonta 4, 30-059 Krak\'ow, Poland,}
\affil{\small ${}^4$ Institute of Mathematics of the Polish Academy of Sciences, \\ 
\'Sniadeckich 8, 00-950 Warszawa,  Poland.} 
\date{}

\maketitle

\begin{abstract}
We give a new definition of dimension spectrum for non-regular spectral triples and compute the exact (i.e. non only the asymptotics) heat-trace of standard Podle\'s spheres $S^2_q$ for $0<q<1$, study its behavior when $q\to 1$ and fully compute its exact spectral action for an explicit class of cut-off functions.
\vspace{1cm}
\end{abstract}
\section{Introduction}
In the classical Riemannian geometry the dimension of the manifold is fixed by its definition and then
later it can be shown to determine some properties of geometric objects. In particular, for closed
manifolds the dimension governs the growth rate of the eigenvalues of the Laplace operator, 
independently of the chosen nondegenerate metric. Unlike the classical situation, in noncommutative 
geometry \cite{Connes,ConnesMarcolli1}, this concept becomes more involved (see \cite{Manin} for a review), 
as there is no intrinsic notion of the dimension. The global description of the noncommutative space in terms 
of an algebra does allow for the possibility of a homological notion of the dimension (relative to Hochschild 
and cyclic (co)homology) or notion of metric dimension based on spectral triples. Yet, the latter is not 
a single number but rather a \textit{dimension spectrum}  \cite{ConnesMoscovici,IochumLevy}, which is 
a (possibly infinite) discrete subset of the complex plane with finite multiplicities allowed.
The dimension spectrum has been computed for various commutative\cite{ConnesMoscovici,IochumLevy,ILS,ConnesMarcolli,Lescure,CGRS} 
and noncommutative spectral triples \cite{GGBISV,EIL,AllPodles,PalSundar,CGIS,GI1,GI2,ILS}
including the quantum group $SU_q(2)$ \cite{ConnesSU2,DLSSV} and some of the Podle\'s spheres
\cite{AllPodles} (see also \cite{Equatorial}). One of the significant features of almost all spectral
triples for the $q$-deformations of manifolds is the discrepancy between the homological and
metric dimensions. Out of these examples, a spectacular one is that of the standard Podle\'s
sphere, with the equivariant real spectral triple of metric dimension $0$.

In the study of Riemannian geometry through spectral properties of the Dirac (or Laplace) operators,  
the main tool is the analysis of the asymptotic expansion of the heat kernel operator. This is well
established for Riemannian manifolds and also for noncommutative geometries with the classical
Dirac operators (like in the case of noncommutative tori or $SU_q(2)$). Yet, the case of standard
Podle\'s sphere offers a new challenge, as the type of operator that appears there has not been
studied extensively.

Our aim in this work, is to determine the exact dimension spectrum (providing a new definition, which
appears to best suit the example under consideration) of the $\UqSU$-equivariant spectral triple introduced in \cite{DabrowskiSitarz}, to compute explicitly the zeta-function relative to the absolute value
of the Dirac operator and provide the representation of the heat-trace $\Tr(e^{-t \vert\DD\vert})$. Surprisingly,
unlike in the classical case, we give not only an asymptotics for $t \to 0$ but an exact representation 
of the heat-trace. This yields also the spectral action $\Tr\big( f(\vert\DD\vert)\big)$ valid for a large class of functions. 
We discuss the properties of the heat-trace and the spectral action, showing 
that the leading parts are not changed when we add a small bounded perturbation to the Dirac operator $\DD \to \DD + A$.

 \medskip

The paper is organized as follows: we first recall the basic ingredients of the spectral triple introduced in \cite{DabrowskiSitarz}. In the next two sections we compute the meromorphic extension of the zeta-function $\zeta_\DD$ and the dimension spectrum of the standard Podle\'s sphere. But since the triple is not regular, we propose a new definition of dimension spectrum (see Definition \ref{newdimspec}) and show its link with the previous one introduced by Connes--Moscovici. Then in Section \ref{action}, we give the explicit and exact heat-trace for $0<q<1$. Moreover,  when $q$ goes to 1 the heat-trace goes to the classical heat-trace for the 2-sphere. The last section is devoted to an exact computation of the spectral action for a specified class of functions with few comments on the fluctuations by one-forms. In the conclusions, we summarize the main differences with previously computed spectral actions.

\section{Spectral triple}

Let us recall the basic facts about the spectral triple over the standard Podle\'s sphere $\Sq$. Its algebra, introduced in \cite{Podles} as an invariant subalgebra of $\A(SU_q(2))$ under a circle action, is the complex $\star$-algebra generated for some parameter $0 < q < 1$ by $A=A^*, B, B^*$ fulfilling the relations
\begin{align}
AB = q^2 BA, && AB^* = q^{-2} B^*A, && BB^* = q^{-2}A(1-A),&& B^*B = A(1- q^2 A).
\end{align}
As a $C^*$-algebra, it is isomorphic to the minimal unitization of the algebra of compact operators
on a separable Hilbert space. Throughout this paper, however, we work with the polynomial 
algebra $\A_q$ in the above generators, which is a dense subalgebra of the latter.

The Hilbert space suitable for a $\UqSU-$equivariant representation of $\A_q$ has an orthonormal basis $\ket{l,m}$ with $m \in \{-l,-l+1, \ldots, l\}$ and $l \in \oh + \N$ and is denoted by $\H_{\oh}$. There are two non-equivalent  $\UqSU-$equivariant representations $\pi_{\pm}$ of $\A_q$ on $\H_{\oh}$  given by the following formulas  \cite{DabrowskiSitarz}: 
\begin{equation}
\begin{aligned}
&\pi_{\pm}(A) \ketQS \vc  A^{+}_{l,m,\pm} \, \ket{\lpo,m}_{\pm} + A^{0}_{l,m,\pm} \, \ketQS 
+ A^{-}_{l,m,\pm} \, \ket{\lmo,m}_{\pm}\,,\\
&\pi_{\pm}(B) \ketQS \vc  B^{+}_{l,m,\pm} \, \ket{\lpo,\mpo}_{\pm} + B^{0}_{l,m,\pm} \, \ket{l,\mpo}_{\pm} 
+ B^{-}_{l,m,\pm} \, \ket{\lmo,\mpo}_{\pm},\\
&\pi_{\pm}(B^*) \ketQS \vc  \wt{B}^{+}_{l,m,\pm}  \ket{\lpo,\mmo}_{\pm} \hspace{-0.1cm}
+ \wt{B}^{0}_{l,m+1,\pm}  \ket{l,\mmo}_{\pm} \hspace{-0.1cm}+ \wt{B}^{-}_{l,m,\pm} \ket{\lmo,\mmo}_{\pm}. 
\end{aligned}
\label{rep}
\end{equation}
where the coefficients $A^j_{l,m,\pm}$, $B^j_{l,m,\pm}$, $\wt{B}^j_{l,m,\pm}$ for $j=-,0,+$ 
are of the following form:
\begin{align}
&A^+_{l,m}  \vc -q^{m+l+\tfrac{1}{2}} \sqrt{ [l-m+1][l+m+1]} \;\alpha^+_l ,\nonumber\\
&A^0_{l,m}  \vc q^{-\oh} \tfrac{1}{1+q^2} \left( [l-m+1][l+m]  - q^2 [l-m][l+m+1]
\right) \alpha^0_l + \tfrac{1}{1+q^2}\,, \label{Arep}\\
& A^-_{l,m}  \vc  q^{m-l-\tfrac{1}{2}} \sqrt{ [l-m][l+m]} \;\alpha^-_l \,,\nonumber\\
\nonumber\\
&B^+_{l,m}  \vc q^{m} \sqrt{ [l+m+1][l+m+2]} \;\alpha^+_l\,, \nonumber\\
&B^0_{l,m}  \vc q^{m} \sqrt{ [l+m+1][l-m]} \;\alpha^0_l\,, \label{Brep}\\
&B^-_{l,m}  \vc q^{m} \sqrt{ [l-m][l-m-1]} \;\alpha^-_l\,,\nonumber\\
\nonumber\\
&\widetilde{B}^+_{l,m}  \vc q^{m-1} \sqrt{ [l-m+2][l-m+1]} \; \alpha^-_{l+1}\,, \nonumber \\
&\widetilde{B}^0_{l,m}  \vc q^{m-1} \sqrt{ [l+m][l-m+1]} \; \alpha^0_l\,, \label{Bsrep}\\
&\widetilde{B}^-_{l,m}  \vc q^{m-1} \sqrt{ [l+m][l+m-1]}\; \alpha^+_{l-1}\,.\nonumber
\end{align}
The coefficients $\alpha_l$ read $\alpha_l^{-} \vc - q^{2l+2} \,\alpha_l^{+}$ and
\begin{align}
\label{alphaP}
\hspace{-4cm}\text{for }\pi_+: \hspace{1cm}&\alpha^0_l \vc \tfrac{1}{\sqrt{q}} \tfrac{ (q - \oq)
[l - \oh][l+ \frac{3}{2}] + q}{[2l][2l+2]}\ , \\
&\alpha^+_l \vc q^{-l-2}
\tfrac{1}{\sqrt{[2l+2]([4l+4]+[2][2l+2])}}\ ;
\end{align}
\begin{align}
\hspace{-4cm}\text{for }  \pi_- : \hspace{1cm}&\alpha^0_l \vc \tfrac{1}{\sqrt{q}} \tfrac{ (q - \oq)
[l - \oh][l+ \tfrac{3}{2}] - q^{-1}}{[2l][2l+2]}\ , \\
&\alpha^+_l \vc q^{-l-1} \tfrac{1}{\sqrt{[2l+2]([4l+4]+[2][2l+2])}}\ .
\label{alphaM}
\end{align}
Here the $q$-numbers are
$$
[n] \vc \tfrac{q^{-n}-q^n}{q^{-1}-q}\,.
$$

The representation $\pi$ is faithful, thus the completion of the polynomial algebra $\A_q$ in the operator norm is the
$C^*$-algebra of the standard Podle\'s sphere.

We are concerned here with the spectral triple $(\A_q,\H,\DD)$ that we describe now: \\
Let $\H \vc \H_+ \oplus \H_-$ where $\H_\pm \vc H_\oh$, with the representation of $\A_q$ on $\H$ and the operator $\DD$ given by 
\begin{align}
&\pi(a) \vc \left(\begin{smallmatrix} \pi_+(a)&0\\0&\pi_-(a)\end{smallmatrix} \right), \label{Defpi} \\
\label{Dirac}
&\DD \vc \left (\begin{smallmatrix} 0 & \bar{w} \,D \\  w \,D & 0 \end{smallmatrix}\right ), \qquad D:\, \ket{l,m}\in \H_\oh \to 
 [l+\tfrac 12] \, \ket{l,m}\in \H_\oh
\end{align}
where $w \in \C \setminus \{0\}$ is an arbitrary constant. 

We remark that 
\begin{align}
\label{AbsDirac}
\vert \DD \vert =\vert w \vert\,\left( \begin{smallmatrix} D &0 \\ 0 & D\end{smallmatrix} \right) \text{ and }\DD = F |\DD| \text{ where } F \vc \tfrac{1}{|\omega|} \bigl( \begin{smallmatrix} 0&\bar{\omega}\\ \omega&0 \end{smallmatrix} \bigr)
\end{align}
satisfies $[F,\,\pi(a)]=0$ for $a \in \A_q$. \\
In the following, $\ket{l,m}_{\pm}$ represents the orthonormal basis $\ket{l,m}$ on the first and second copy of $\H_\oh$ in $\H$. 
In particular, $\ket{l,m}_+\oplus \ket{l,m}_-$ is an eigenvector of $\vert\DD\vert$ with eigenvalue 
$\vert w \vert [l+\half]$ and
\begin{align}
\label{q_num}
[l+\half]=\tfrac{q^{-l+\half}(1-q^{2l+1})}{1-q^2}\,.
\end{align}
The spectral triple $(\A_q,\H,\DD)$ is $\Z_2$-graded by $\gamma \vc \left( \begin{smallmatrix}1&0\\0&-1 \end{smallmatrix} \right)$ and is real for the antiunitary operator $J$ on $\H$ defined by 
$$
J \, \ket{l,m}_{\pm} \vc i^{2m}\,\ket{l,-m}_{\mp}\,\, \text{ with }p\in \R^+.
$$
In particular, $J^2=-1,J\gamma=-\gamma J$, $JaJ^{-1}$ commutes with $\A_q$ and $[\DD,J]=0$, so the spectral triple is of $KO$-dimension 2 \cite{Projective}.

Any generator is in the representation \eqref{rep} a sum of three weighted shifts, written in
a shorthand notation as $\pi_\pm(a) = \pi_\pm^+(a) + \pi_\pm^0(a) + \pi_\pm^-(a)$ so $\pi \vc \pi^+ + \pi^0 + \pi^-$ using \eqref{Defpi}.\\
In the following, to alleviate notations, we often replace $\pi(a)$ by $a$ for $a\in \A_q$.

\begin{lemma} {\rm (Compare \cite{SitarzSq})}
\label{Dec}
For any $a\in \A_q$, the decomposition via the weighted shifts $\pi=\pi^+ +\pi^0 +\pi^-$ is such that $[\DD,\pi^0(a)]=0$ and $\DD\,\pi^{\pm}(a)$ is a bounded operator.
\end{lemma}

\begin{proof}
For any generator $X=A,B,B^*$, and thanks to (\ref{Arep}--\ref{alphaM}), we get $X^{\pm}_{l,m} = \OO(q^l)$ while the eigenvalues of $\DD$ behave like $[l+\tfrac{1}{2}]=\OO(q^{-l})$. The bounded part $\pi^0(a)$ (for instance $A^0_{l,l}=\OO(1)$) is diagonal so commutes with $\DD$.
\end{proof}

The operator $\DD$ is the unique $\UqSU-$equivariant one which makes the spectral triple real \cite{DabrowskiSitarz}.

\section{Dimension spectrum}

We first compute the zeta-function $\zD(s) \vc \Tr |\DD|^{-s}$ associated to $\DD$, see \cite{ConnesLetter,ConnesMarcolli1}:

\subsection{\texorpdfstring{The zeta-function $\zeta_\DD$}{The zeta-function}}

\begin{Prop}
\label{obs}
The spectral zeta-function $\zD(s)$ is given by
\begin{align}
\zD(s) & = 4 (\tfrac{1-q^2}{\vert w \vert})^{s} \,\sum_{k = 0}^{\infty} (k + 1)\, \tfrac{q^{ks}}{\left(1 - q^{2(k+1)} \right)^{s}}, \quad \text{for } \, \Re(s) >0 \,.
\label{zeta_D}
\end{align}
\end{Prop}

\begin{proof}
This is a direct computation:
\begin{align*}
\Tr |\DD|^{-s}  &= \sum_{\{+,-\}} \sum_{l \in \N + \oh} \sum_{m = -l}^{m = l} \braQS |\DD|^{-s} \ketQS   
= \sum_{\{+,-\}} \sum_{l \in \N + \oh} \sum_{m = -l}^{m = l} 
\big(\vert w \vert \, \tfrac{q^{-l+\half}(1-q^{2l+1})}{1-q^2}\big)^{-s} \\
&= 2\sum_{l \in \N + \oh}  (2 l + 1) \big(\vert w \vert \,\tfrac{q^{-l+\half}(1-q^{2l+1})}{1-q^2}\big)^{-s} 
= 4\sum_{k = 0}^{\infty}  (k + 1) \big(\vert w \vert \,\tfrac{q^{-k}(1-q^{2(k+1)})}{1-q^2}\big)^{-s} .
\end{align*}

Moreover, since $q^{2(n+1)} \leq q^2 \, \Rightarrow \, 0 \leq (1-q^{2(n+1)})^{-x} \leq (1-q^2)^{-x}, \text{ for any }x>0$, 
we have, for $s= x + i y$ with $x>0$,
\begin{align*}
|\zD(s)| & \leq 4 (\tfrac{1-q^2}{\vert w \vert})^{x} \,\sum_{k = 0}^{\infty} (k + 1)\, \tfrac{q^{kx}}{(1 - q^{2(k+1)})^{x}} \leq  \tfrac{4}{{\vert w \vert}^x} \,\sum_{k = 0}^{\infty} (k + 1)\, q^{kx} = \tfrac{4}{\vert w \vert^x \, \left(1-q^x\right)^2} < \infty.
\end{align*}
So \eqref{zeta_D} holds for any $s$ with $\Re(s) > 0$.
\end{proof}

From equality \eqref{zeta_D} we deduce directly
\begin{align}
\vert\Tr\,(\vert \DD \vert^{-s}) \vert \in \R^+,  \text{ for } \Re(s)>0, \label{traceable}
\end{align}
meaning that the metric dimension $d=\inf \{d' >0 \,\,\vert\,\, \vert \Tr\,(\vert \DD \vert^{-d'}) \vert <\infty\}$ of the 
spectral triple $(\A_q,\H,\DD)$ is zero.

To compute the full dimension spectrum of the spectral triple we need a meromorphic extension of the zeta-function $\zD$.

\begin{Prop}\label{PropZeta}
The spectral zeta-function of the standard Podle\'{s} sphere admits a meromorphic continuation to the whole 
complex plane given by the formula
\begin{align}
\label{zeta_D1}
\zD(s) = 4 (\tfrac{1-q^2}{\vert w \vert})^{s} \sum_{n=0}^{\infty} \, \tfrac{\Gamma (s+n)}{n! \, \Gamma (s)} \,
\tfrac{q^{2n}}{(1-q^{s+2n})^2}\,.
\end{align}
All poles of $\zD$ are of the second order and form an infinite countable subset $\Sd_1$ of 
the complex plane:
\begin{align}\label{S}
\Sd_1 \vc -2\N + i\tfrac{2 \pi}{\log q} \,\Z \,.
\end{align}
\end{Prop}

\begin{proof}
Since $(1-q^l)^{-s}=\sum_{n=0}^\infty \tfrac{\Gamma(s+n)}{n! \, \Gamma(s)} \,q^{ln}$ for any integer $l$, equation 
\eqref{zeta_D1} follows by changing order of summation in series, 
\begin{align*}
\zD(s) & = 4 (\tfrac{1-q^2}{\vert w \vert})^{s} \, \sum_{k = 0}^{\infty} (k + 1)\, q^{ks} \sum_{n=0}^\infty \tfrac{\Gamma(s+n)}{n! \, \Gamma(s)} \,q^{2(k+1)n}  = 4 (\tfrac{1-q^2}{\vert w \vert})^{s} \, \sum_{n=0}^\infty \tfrac{\Gamma(s+n)}{n! \, \Gamma(s)}\, q^{2n} \, \sum_{k = 0}^{\infty} (k + 1)\, q^{k(2n+s)} \\
& = 4 (\tfrac{1-q^2}{\vert w \vert})^{s} \, \sum_{n=0}^\infty \tfrac{\Gamma(s+n)}{n! \, \Gamma(s)} \tfrac{q^{2n}}{(1-q^{s+2n})^2}\,.
\end{align*}

To show that there are no poles of $\zD$ in $\C$ outside of $\Sd_1$ let us first take
\begin{align*}
\Re(s) = -2M - \delta, \quad \text{with } M \in \N, \, \delta \in (0,2).
\end{align*}
\begin{align*}
&\text{For }n \geq M+1, \text{ we have}\\
&\qquad \Re(s) + 2n \geq 2 - \delta \, \Rightarrow \, 1- q^{\Re(s) + 2n} \geq 1 - q^{2-\delta} \geq 0 \, \Rightarrow \, (1- q^{\Re(s) + 2n})^2 \geq (1 - q^{2-\delta})^2,\\
&\text{whereas for }n \leq M, \\
&\qquad \Re(s) + 2n \leq - \delta \, \Rightarrow \, 1- q^{\Re(s) + 2n} \leq 1 - q^{-\delta} \leq 0 \, \Rightarrow \, (1- q^{\Re(s) + 2n})^2 \geq (1 - q^{-\delta})^2.
\end{align*}
So for any $n$,
\begin{align*}
(1- q^{\Re(s) + 2n})^2 \geq c(q,\delta), \quad \text{with } c(q,\delta) = \min \{ (1 - q^{2-\delta})^2, (1 - q^{-\delta})^2 \}.
\end{align*}
Moreover,
\begin{align}
\label{module}
|1-q^{s+2n}|^2 = 1 -2 q^{\Re(s) + 2n} \cos\big( \Im(s) \log q \big) + q^{2 \Re(s) + 4n} \geq (1- q^{\Re(s) + 2n})^2 \geq c(q,\delta),
\end{align}
so that
\begin{align} \label{zDestim}
|\zD(s)| & \leq 4 c(q,\delta)^{-1} \, (\tfrac{1-q^2}{\vert w \vert})^{\Re(s)} \, \sum_{n=0}^\infty \tfrac{q^{2n}}{n!} \, \left\vert \tfrac{\Gamma(s+n)}{\Gamma(s)} \right\vert \leq 4 c(q,\delta)^{-1} \, (\tfrac{1-q^2}{\vert w \vert})^{\Re(s)} \, (1-q^2)^{-\vert s\vert} <\infty,
\end{align}
since
\begin{align*}
\sum_{n=0}^\infty \tfrac{q^{2n}}{n!} \, \left\vert \tfrac{\Gamma(s+n)}{\Gamma(s)} 
\right\vert = \sum_{n=0}^\infty \tfrac{q^{2n}}{n!} \, \left\vert (s+n)(s+n-1)\cdots(s+1)s 
\right\vert \leq \sum_{n=0}^\infty \tfrac{q^{2n}}{n!} \, \tfrac{\Gamma(\vert s\vert+n)}{\Gamma(\vert s\vert)}=(1-q^2)^{-\vert s\vert}.
\end{align*}

Now, assume
\begin{align*}
\Re(s) = -2M, \quad \Im(s) = \tfrac{2\pi}{\log q} (N+\epsilon), \quad \text{with } M \in \N, \, N \in \Z, \, \epsilon \in (0,1).
\end{align*}
By the same reasoning as above we have
\begin{align*}
&\text{for } n \geq M+1, && |1-q^{s+2n}|^2 \geq (1- q^{-2M + 2n})^2 \geq (1- q^{2})^2, \\
&\text{for } n \leq M-1, && |1-q^{s+2n}|^2 \geq (q^{-2M + 2n}-1)^2 \geq (q^{-2}-1)^2 \geq (1- q^{2})^2, \\
&\text{and for } n = M, && |1-q^{s+2n}|^2 = 2 -2 \cos( \Im(s) \log q ) = 2 -2 \cos( 2 \pi \epsilon ).
\end{align*}
Thus for any $n$,
$|1-q^{s+2n}|^2 \geq \wt{c}(q,\epsilon), \text{ with }  \wt{c}(q,\epsilon) = \min \{ (1 - q^{2})^2, 2 -2 \cos( 2 \pi \epsilon ) \},
$
and again 
$|\zD(s)|  \leq 4 \wt{c}(q,\epsilon)^{-1} \, (\tfrac{1-q^2}{\vert w \vert})^{\Re(s)} \, \sum_{n=0}^\infty \tfrac{q^{2n}}{n!} \, \left\vert \tfrac{\Gamma(s+n)}{\Gamma(s)} \right\vert <\infty$.

We have thus shown that the poles of $\zD$ come uniquely from the coefficients $(1-q^{s+2n})^{-2}$, so are indeed of second order and located at $s \in \Sd_1$.
\end{proof}

Before we end this section let us make a side-remark about the $q$-zeta functions and their applications. The $q$-deformed special functions where already considered by Ramanujan \cite{Ramanujan} in his famous notebook.
There were various different propositions in the 
literature \cite{Satoh,Ueno_Nish,Cherednik,Kaneko} for a $q$-analogue of the Riemann zeta-function including the generalizations to multiple $q$-zeta-functions 
\cite{Bradley,Zhao} and Euler $q$-zeta-functions \cite{Kim}. It is worth mentioning that, in \cite{Ueno_Nish}, the 
motivation for studying Riemann $q$-zeta-functions came precisely from the theory of quantum groups. 
Unfortunately, to our knowledge, this direction has not been further investigated and the main interest in 
$q$-zeta-functions was in number theory \cite{Cherednik,Kaneko}.

An important result on meromorphic extension of $q$-zeta function has been obtained in \cite{Kaneko}, with a slightly different definition of $q$-numbers $\{x\}_q \vc \tfrac{1-q^x}{1-q}$. In fact, an alternative proof of Proposition \ref{PropZeta} could be given by observing that
\begin{align*}
\zD(s) =\tfrac{2}{\log q} \, (\vert w \vert q)^{-s} \,\tfrac{\dt}{\dt \, t} f_{q^2}(s,t) |_{t= s/2}\,,
\end{align*}
with $f_q(s,t) = \left( 1-q \right)^{s} \sum_{r=0}^{\infty}  \tfrac{\Gamma (s+r)}{\Gamma (r+1)\Gamma (s)} \, \tfrac{q^{t+r}}{1-q^{t+r}}$ being the meromorphic extension of a generalized Riemann $q$-zeta function \cite[Proposition 1]{Kaneko}.

Finally, let us mention a result \cite[Proposition 4.1]{SeniorKaad} analogous to our Proposition \ref{PropZeta} that appeared in the context of $SU_q(2)$ quantum group.

\subsection{About the dimension spectrum}
Let $\B(\H)$ be the algebra  of bounded operators.\\
We now recall the definition of the dimension spectrum:
\begin{Def}
\label{dimspec}
{\rm (see \cite{ConnesMoscovici})} 
Given a spectral triple $(\A,\H,\DD)$, let $\delta$ be the derivation defined on $\B(\H)$ by $\delta(T)\vc [\vert \DD \vert, T]$ (so $T$ preserves $\Dom(\vert \DD\vert)$ and $[\vert\DD\vert,T]$ extends (uniquely) in $\B(\H)$).\\
For $\alpha \in \R$, define the set of operators of order (at most) $\alpha$ by 
$$\OP^\alpha \vc \{T\in \mathcal{L}(\H) \text{ such that }\vert \DD \vert^{-\alpha} T \in \OP^0\} \text{ where }\OP^0 \vc \bigcap_{n\in\N} \Dom\,\delta^n.
$$
Assuming $\A \subset \OP^0$, let $\B_0$ be the algebra generated by $\delta^n (a)$ for $n\in \N$ and $a \in \A$ (where by a standard abuse of notation we shall omit the symbol of representation $\pi$), 

A spectral triple $\ahd$ has dimension spectrum $\Sd$ if $\Sd \subset \C$ is discrete and for any element 
$P\in \B_0$, the function
\begin{align}\label{zetaB}
\zeta_\DD^P(s) \vc \Tr (P |\DD|^{-s})
\end{align}
extends holomorphically to $\C \setminus \Sd$.
\end{Def}
Let us stress here that the above original definition of dimension spectrum requires a kind of ``weak regularity'' condition: $\A \subset \OP^0$. To define the pseudodifferential calculus, one imposes a stronger regularity assumption \cite{ConnesMoscovici} on the spectral triple. It consists of requiring that not only $\A \subset \OP^0$, but also $[\DD,\A] \subset \OP^0$. Under these assumptions, the algebra of pseudodifferential operators $\Psi DO(\A)$ is the set of all operators $T$ of the form
\begin{align}\label{PDO}
T \simeq \sum_{n\in\N} P_n\,\vert \DD\vert^{d-n}
\end{align}
where the $P$'s are in the polynomial algebra $\DD(\A)$ generated by $\A, \,J\A J^{-1}, \,\DD$ and $|\DD|$ and moreover, the $\simeq $ means that $T= \sum_{n=0}^N P_n\,\vert \DD\vert^{d-n} \mod(\OP^{-N})$ for each $N\in \N$.

However $\A_q \nsubseteq \OP^0$: it is known \cite{NeshTuset,KrahmerWagner} that the $\Uq$-equivariant spectral triple on standard Podle\'s sphere does not satisfy the regularity assumption. In fact it does not even satisfy the weak regularity condition. Indeed, already $\big[|\DD|,[|\DD|,a] \big]$ is an unbounded operator and more generally $\delta^n(a)$ is an operator of order $n-1$ for a generic element $a \in \A$ and this could be seen by computing a fixed matrix element of the operator $\delta^n(A)$, for instance: 
\begin{align*}
{}_{\pm} \bra{l+1,m} \, \delta^n(A) \ket{l,m}_{\pm} & = \left( [l+\tfrac{3}{2}] - [l+\tfrac{1}{2}] \right)^n A^+_{l,m}\,.
\end{align*}
Since the behaviour for large $l$ of $A^+_{l,m}$ in (\ref{Arep}) is like  $\OO(q^l)$ while $\left( [l+\tfrac{3}{2}] - [l+\tfrac{1}{2}] \right)^n =\OO(q^{-nl})$, the above expression is unbounded for $n>0$, and, generally, $\delta^n(A)$ is in $\OP^{n-1}$. Note that $\delta^0(a)$ and $\delta^1(a)$ are bounded for any $a \in \A_q$, but not $\delta^2(a)$.

Hence, if one insists on keeping the original definition of the dimension spectrum without the regularity property, then one would find in $\B_0$ operators of arbitrary order. As a consequence, the real part of the dimension spectrum would 
not be bounded from above and one would be forced to conclude that $\Sd(\Sq) = \Z + i\tfrac{2 \pi}{\log q} \Z$. 
This would mean that the metric dimension (viewed as a largest real number in $\Sd$) of the spectral triple at 
hand is infinite.

Let us recall that in \cite{ConnesMoscovici} is defined another class of operators: we get the following scale of spaces for a parameter $s \in \R$ 
\begin{align}
\label{Hs}
&\H^s\vc \Dom(|\DD|^s)=\Dom \big((1+\DD^2)^{s/2}\big).
\end{align}
(The second equality is used when $\DD$ is not invertible.) 
If $s\leq 0$, $\H^s=\H^0=\H$ and for $s\geq 0$, they are Hilbert spaces for the Sobolev norm $\norm{\xi}_s^2 \vc \norm{\xi}^2 +\norm{(1+\DD^2)^{s/2} \xi}^2$, with $\xi \in \H^s$. Moreover, $\H^{\infty} \vc \underset{s \geq 0}{\cap} \H^s \text{ and } \H^{-\infty} \vc \text{topological dual of } \H^{\infty}$, 
\begin{align*}
\op^r \vc &\set{ T\in \Lc(\H) \,\vert \,\Dom(T) \supset \H^\infty \text{ and T maps }(\H^\infty, \norm{\cdot}_s) \text{ continuously into } (\H^\infty, \norm{\cdot}_{s-r}) \\
& \quad\text{ for each } s \geq r}.
\end{align*}
Thus any operator in $op^r$ has a bounded extension from $\H^s \to \H^{s-r}$ for any $s\geq r$ and $\op^0 \subset \B(\H)$ since $\H^0=\H$. In particular $\op^r \subset \op^s$ when $r\leq s$ and $\vert \DD\vert^r \in \op^r$ for any $r \in \R$.  Moreover $\OP^r \subset \op^r$ for any $r \in \R$ since $\op^r\cdot \op^s \subset \op^{r+s}$ and $\bigcap_{n\in\N} \Dom\,\delta^n \subset \op^0$ (compare \cite[Corollary 6.3]{CPRS}).

Our results formulated below show that although the regularity is not satisfied, a pseudodifferential calculus on the $\Uq$-equivariant spectral triple on standard Podle\'s sphere is possible. 

\begin{Lm}
\label{lemma1as} {\rm (Compare \cite[Corollary 3.3]{NeshTuset})} 
For any element $a,b \in \A_q$ and any $z \in \C$ we have 
\begin{align}
&|\DD|^z a= a |\DD|^z + {B}(a,z) \,|\DD|^{z-1}, \label{Dza}\\
&|\DD|^{z} db = \chi^z \,db \,|\DD|^{z} + \wt B(b,z) \,|\DD|^{z-1}, \quad z\in \C \label{qmutator}
\end{align}
where 
$$
\chi \vc \left (\begin{matrix} q & 0 \\ 0 & q^{-1} \end{matrix}\right )
$$
and $z \mapsto B(a,z) ,\,\wt{B}(a,z)$ 
are two analytic functions in $z$ with values in bounded operators on $\H = \H_{+} \oplus \H_{-}$ for any $a\in \A_q$. 
In particular, $a$ and $[\DD,a]$ preserve $\H^\infty$ for all $a\in \A_q$.
\\ More generally, for any one-form $\Ag\vc \sum_{finite} a_idb_i$ and any $z \in \C$, we have
\begin{align}
& |\DD|^z \Ag = \chi^z \Ag |\DD|^z + \widehat{B}(\Ag,z)|\DD|^{z-1}, \label{com}\\
& \DD \Ag = \chi^{-1} F \Ag |\DD| + \chi^{-1} F \widehat{B}(\Ag,1), \label{DA}
\end{align}
with an analytic function $z \mapsto \widehat{B}(\Ag,z)$ valued in $\B(\H)$ for any $\Ag$.
\end{Lm}

\begin{proof} 
From Lemma \ref{Dec} we know that $|\DD|^z \,\pi^0(a) = \pi^0(a) \,|\DD|^z$ and that $|\DD|\, \pi^\pm(a)$ is a bounded operator. Moreover, we have
\begin{align*}
|\DD|^z \,\pi^\pm(a) - \pi^\pm(a)\, |\DD|^z  =:  \left(|\DD|^z - |\DD^\pm|^z \right) \pi^\pm(a), 
\end{align*}
where $|\DD^\pm|$ is essentially the module of a shifted Dirac operator, that is a $\Z_2$-graded
operator with the eigenvalues $\pm |w| [l+ \tfrac{1}{2} \pm 1]$. Our aim is to show that the above 
expression could be written as $B(\pi^\pm(a),z) |\DD|^{z-1}$ for an analytic function $B(\pi^\pm(a),z)$ in $\B(\H)$:
\begin{align*}
|\DD|^z \,\pi^\pm(a) - \pi^\pm(a)\, |\DD|^z &=
\left(|\DD|^z - |\DD^\pm|^z \right) |\DD^\pm|^{-z+1} \pi^\pm(a) |\DD|^{z-1} \\
&= \left(\left(|\DD|^z - |\DD^\pm|^z \right) |\DD^\pm|^{-z} \right)\, \big(|\DD^\pm|  \, \pi^\pm(a) \big)\, |\DD|^{z-1} \\
& =: B(\pi^\pm(a),z) |\DD|^{z-1}.
\end{align*}
Indeed, the first bracket is a diagonal operator with eigenvalues
$$\big([l+\tfrac{1}{2}]^z-[l+ \tfrac{1}{2} \pm 1]^z\big) \big([l+\tfrac{1}{2} \pm 1] \big)^{-z}$$
 and for any $z \in \C$, it is bounded and analytic as a function of $z$. On
 the other hand, the element $|\DD^\pm|\, \pi^\pm(a)$ is bounded, as we observed earlier, hence
 the conclusion.

The second equality \eqref{qmutator} is obtained in \cite[Corollary 3.3]{NeshTuset} and the equality \eqref{com} follows from the two previous one. Using $z=1$ in \eqref{com} together with $F\chi=\chi^{-1}F$ we obtain 
\begin{align*}
\DD \Ag = F |\DD| \Ag = F \chi \Ag |\DD| + F \chi \widehat{B}(\Ag,1) = \chi^{-1} F \Ag |\DD| + \chi^{-1} F \widehat{B}(\Ag,1).
\end{align*}
\end{proof}

\begin{Prop}
\label{op0}
We have $\A_q \subset \op^0$ and $[\DD,\A_q] \subset \op^0$. Moreover $J,J^{-1} \in \op^0$.
\end{Prop}

\begin{proof}
The assertion follows directly from Lemma \ref{lemma1as} which allows us to perform commutations of $|\DD|$ with one-forms (given by finite sums of $adb \vc a[\DD,b]$) up to operators of lower order and the fact $\H^{s} \subset \H^{s-1}$ for any $s \in \R$. Since $J$ commutes with $\DD$, it preserves $\H^\infty$, so $J \in \op^0$.
\end{proof}

To construct a pseudodifferential calculus, we need one additional property.

\begin{lemma}
\label{lemma2as}
For any $a \in \A$ and $z \in \C$, we have
\begin{align*}
 \chi^z \, a = a \,\chi^z , \quad \chi^z \,\DD = \DD\, \chi^{-z},\quad \chi^z \,da = da \,\chi^{-z}.
\end{align*}
\end{lemma}

\begin{proof}
The lemma just uses the fact that $\chi$ commutes with the chirality-preserving operators ($T : \H_{\pm} \to \H_{\pm}$) and
\begin{align*}
\left( \begin{smallmatrix} q^z & 0 \\ 0 & q^{-z} \end{smallmatrix} \right) \left(\begin{smallmatrix} 0 & T_1 \\ T_2 & 0 \end{smallmatrix}\right) = \left(\begin{smallmatrix} 0 & T_1 \\ T_2 & 0 \end{smallmatrix}\right ) \left (\begin{smallmatrix} q^{-z} & 0 \\ 0 & q^{z} \end{smallmatrix}\right ).
\end{align*}
\end{proof}

Lemmas \ref{lemma1as} and \ref{lemma2as} guarantee that we can commute the elements of the algebra generated by $a$, $J a J^{-1}$ and $[\DD,a]$, $a \in \A_q$, through $|\DD|$ without changing its $\op$ class. The commutations through $J$ and $F$ neither would change the class. It thus justified to define the order of an operator on $\H$ in the non-regular case to be its $\op$ class instead of the standard approach based on $\OP$ classes.

We shall keep the standard definition of the algebra of pseudodifferential operators \cite{ConnesMoscovici,IochumLevy}
\begin{align*}
\Psi(\A) \vc \big\{ P \; \big| \; \forall \; N \in \N, \quad \exists \; T \in D(\A), \, R \in \OP^{-N}, \, p \in \N, \quad P = T|\DD|^{-p} + R \big\},
\end{align*} 
with $D(\A)$ being a polynomial algebra generated by $\A, \DD, |\DD|$ (and $J \A J^{-1}$ in the real case).
On the other hand we change the notion of a pseudodifferential operator of order $k$ for
\begin{align} \label{NewOrder}
\Psi^k(\A) \vc \Psi(\A) \cap \op^k.
\end{align}

The lemmas of this section prove that $\Psi(\A)$ is still a graded algebra with the grading defined by \eqref{NewOrder} and that $\PDO$'s of complex order are sound and controllable. 

Finally, let us remark that in the regular case $|\DD|^{-1} \in \Psi(\A) \cap \OP^0$, but $|\DD|^{-1} \notin \B_0$ and that, for a $d$-dimensional compact spin Riemannian manifold, the dimension spectrum of Definition \ref{dimspec} is only included in $\set{d-n \,\vert \,n\in \N}$ while to get equality one needs the whole set of $\Psi^0(\A)$ \cite{IochumLevy}.

The above remarks encourage us change previous Definition \ref{dimspec} of dimension spectrum as:
\begin{definition}
\label{newdimspec}
Let $(\A,\H,\DD)$ be a spectral triple such that $\A$ and $[\DD,\A]$ are in $op^0$ and let $\B \vc \Psi^0(\A) = \Psi(\A) \cap \op^0$. The dimension spectrum $\Sd$ is a discrete subset of $\C$ defined as the poles of $\zD^P: s \mapsto \Tr(P |\DD|^{-s})$ for $P\in \B$.
\end{definition}
Note that $\B$ contains the algebra $\B_1$ of all operators of the form
 \begin{align}
 \label{b}
|\DD|^{z}  \omega |\DD|^{-z-p} ,
\end{align}
for any $z \in \C$, $p\in \N$ and any $\omega = a_0 d a_1 \cdots da_k$ where $k \geq 0$ and $a_j\in \A\cup J\A J^{-1}$, since $\op^r\cdot\op^s\subset \op^{r+s}$.
The algebra $\B$ could be seen as a the largest algebra of generalized pseudodifferential operators of order at most $0$. 

Remark that, when $(\A,\H,\DD)$ is regular ($\A$ and $[\DD,\A]$ are in $\OP^0=\bigcap_{n\in\N} \Dom\,\delta^n$), this definition extends Definition \ref{dimspec}: with regularity, the algebra $\B_0$ generated by the $\delta^n(\A)$ is included in $\OP^0$ and contained in $\B$. Moreover, since $\vert \DD\vert \,b \,\vert\DD\vert^{-1} - b =[\vert \DD\vert,b] \,\vert \DD\vert^{-1}$ for any $b\in \A \cup [\DD,\A]$ (by \cite[Lemma B.1]{ConnesMoscovici} $\vert \DD\vert \,b \,\vert\DD\vert^{-1} - b  \in op^{-1}$), we see that $\B_0$ is equal to $\B$ up to $\vert \DD \vert^{-1}\in \B_1 \subset \B$. In the regular case, the differences between $\B_0$ and $\B_1$ are not essential since they are subsets of pseudodifferential operators. 

Note that in \cite[p.206]{ConnesMoscovici} it is also assumed that the analytic continuation of $\zeta_\DD^P$ is such that $\Gamma(z)\,\zeta_\DD^P(z)$ is of rapid decay on vertical lines $z = x+ iy$, for any $x > 0$. This technical assumption was introduced to facilitate the use of heat kernel methods in the proof of local index index theorem \cite[p.217]{ConnesMoscovici} (see also \cite{NeshTuset}). It would not be satisfied for the standard Podle\'s sphere with the definition \eqref{dimspec} of $\B_0$, since unbounded elements would give rise to poles at $z = n + \tfrac{2 \pi}{\log q} i$ for $n \in \N_+$. However, with the new definition of $\B$, this assumption is satisfied. In the proof of Proposition \ref{MellinLemma2} we show that the contour integration is sound since $|\Gamma(x \pm i y) \,\zD(x \pm i y)|$ decays exponentially with $y$ for any value of $x$. On the other hand, from the proof of Proposition \ref{DimSp} it follows that the large $y$-behaviour of $\zeta_\DD^P(x\pm i y)$ is not altered by $P \in \B$ for $x >0$.

Now we are in a position to compute the dimension spectrum of the $\Uq$-equivariant triple on the standard Podle\'s sphere. 

\begin{Prop}
\label{DimSp}
When $P \in \B$, the set of poles of the function $\zD^P: s \mapsto \Tr(P |\DD|^{-s})$ is a 
subset of $-\N + i\frac{2 \pi}{\log q} \Z$ and the poles are at most of the second order.
\end{Prop}

{\it Proof.} 
It is sufficient to consider 
\begin{align*}
P =  T |\DD|^{-p},
\end{align*}
with $T \in D(\A) \cap \op^n$ for some $n \in \N$ and $\N \ni p \geq n$. Applying a finite number of times the Lemmas \ref{lemma1as} and \ref{lemma2as} we can move all the operators $|\DD|$ and $\DD$ to the right hand side while moving all $\chi$ to the left. In the end we obtain
\begin{align} \label{P}
P =&  \sum_{k=0}^{N} \chi^{r(k)} h_k(T; q) \, |\DD|^{n-p-k},
\end{align}
where $h_k$ is a bounded operator on $\H$ and $r(k) \in \Z$ for any  $k \in \{0, \ldots, N\}$.

Since all generators of the algebra are represented 
by weighted shifts on the Hilbert space and $|\DD|$ is a diagonal operator, the operator $h_k$ must be again a finite sum of 
weighted shifts. On the basis vector $\ketQS $ of $\H_{\pm}$ it assumes the form for $r=0,1$
\begin{align*}
h_k\, \ket{l,m}_{(-)^r} = \sum_{r' = 0}^{1} \,\sum_{i= -M}^{M}\, \sum_{j=-N}^{N} P_{h_k}(l, m, r, r', i, j) \,\ket{l+i,m+j}_{(-)^{r'}}\,,
\end{align*}
where both the $M$, $N$ and the function $P_{h_k}$ depends on the arguments of the operator $h_k$. 
Since our aim is to compute $\Tr (P |\DD|^{-s})$, we are only be interested in the diagonal part of 
the operator $h_k$. Moreover, since $\zD^P(s) = \zD^{P_0}(s-n+p)$ with $$P_0 \vc P |\DD|^{-n+p},$$
we restrict to the case $p = n$.

Let us abbreviate $P_{h_k}(l, m, r) \vc P_{h_k}(l, m, r, r, 0, 0)$. Now, for any $h_k$, the function 
$P_{h_k}$ will be a finite composition of polynomials, rational functions and square roots of (bounded) 
variables $q^{l+m}, \,q^{l-m}$ and $q^l$ as can be deduced from the explicit forms of the 
representation  (\ref{Arep}--\ref{Bsrep}) and the Dirac operator (\ref{Dirac}). 

Moreover, due to reality of the coefficients 
$A^j_{l,m,\pm}$, $B^j_{l,m,\pm}$, $\wt{B}^j_{l,m,\pm}$, one can express the square roots in terms of infinite series using the formulas
\begin{align*}
\sqrt{[l \pm m]}  &= (q^{-1}-q)^{-1/2} q^{-(l\pm m)/2} \sqrt{1-q^{2(l \pm m)}} \\
&= - (q^{-1}-q)^{-1/2} \sum_{n=0}^{\infty} \tfrac{(2n)!}{4^n (n!)^2 (2n-1)} \, q^{(2n + 1/2)(l \pm m)}, 
\end{align*}
which are convergent for all $l \in \N+\oh$, $-l \leq m \leq l$ and $0<q<1$.

An analogous operation can be done for rational functions involved in the formulas 
(\ref{alphaP}--\ref{alphaM}). Since a finite sum 
and a finite product of convergent power series is again a power series with at least finite non-zero 
radius of convergence (such radius appears in next lemma), we may use the formula 
\begin{align}
\label{pf}
P_{h_k}(l,m,r) = \sum_{\alpha,\beta = 0}^{\infty} \wt{p_h}(\alpha,\beta,r) 
\,q^{\alpha (l+m)} \,q^{\beta (l-m)},
\end{align}
which holds for $l \in \N+\oh$, $-l \leq m \leq l$ and $0<q<1$.

Let us now use the above considerations to investigate the holomorphic structure of functions 
$\zD^{P_0}$. Since the sum over $k$ of \eqref{P} is finite it is sufficient to consider the function $\zD^{P,k}(s) \vc \Tr \big( h_k(T;q) |\DD|^{-k-s} \big)$. We have, 
assuming $\vert w \vert = 1$ to simplify,
\begin{align*}
\zD^{P,k}(s) &= \Tr (h_k |\DD|^{-s}) = \sum_{r = 0}^{1}\, \sum_{l \in \N + \oh} \,\sum_{m=-l}^{l} 
P_{h_k}(l,m,r) [l+\tfrac{1}{2}]_q^{-k-s} \\
 &= \sum_{r = 0}^{1} \,\sum_{l \in \N + \oh}  \,\sum_{m=-l}^{l} \,\sum_{\alpha,\beta = 0}^{\infty} 
\wt{P_{h_k}}(\alpha,\beta,r)\, q^{\alpha (l+m)} q^{\beta (l-m)} \left( 1-q^2 \right)^{s+k} 
\tfrac{q^{(s+k) (l-1/2)}}{\left(1 - q^{2l+1)} \right)^{s+k}} \,.
\end{align*}

Since for $\Re(s) > 0$, all of the series in the above expressions are convergent, we can perform 
the partial summation over $m$, so
\begin{align}
\zD^{P,k}(s)&=\sum_{r = 0}^{1} \,\sum_{l \in \N+\oh}  \,
\sump{\alpha,\beta= 0}{\infty} \,
\wt{P_{h_k}}(\alpha,\beta,r) \,q^{(\alpha+\beta) l} \,\tfrac{q^{(\beta-\alpha)l} - q^{(\alpha-\beta)(l+1)}}
{1-q^{\alpha-\beta}} \left( 1-q^2 \right)^{s+k} \tfrac{q^{(s+k) (l-1/2)}}{\left(1 - q^{2l+1} \right)^{s+k}}  \nonumber \\
& \hspace{0.5 cm}\quad + \sum_{r = 0}^{1} \,\sum_{l \in \N+\oh}^{\infty} \,\sum_{\alpha = 0}^{\infty} 
\wt{P_{h_k}}(\alpha,\alpha,r) \, (2l+1) \, q^{2\alpha l} \, \big( 1-q^2 \big)^{s+k} 
\tfrac{q^{(s+k) (l-1/2)}}{\left(1 - q^{2l+1} \right)^{s+k}}\,, \label{zDb_abc}
\end{align}
where ${\sum'}_{\alpha,\beta = 0}^{\,\infty}$ denotes the sum with the diagonal indices 
$\beta=\alpha$ omitted. 

Let us now make use of the analytic continuation of $\zD$ in \eqref{zeta_D1} to define an analytic continuation of $\zD^{P,k}(s)$:
\begin{align}
\zD^{P,k}(s)&=\left( 1-q^2 \right)^{s+k} \sum_{r = 0}^{1}\, 
\sump{\alpha,\beta = 0}{\infty} \, \sum_{j=0}^\infty \, \tfrac{\Gamma(s+k+j)}{j! \, \Gamma(s+k)} \,  \tfrac{\wt{P_{h_k}}(\alpha,\beta,r) \, q^{2 j} \, \left[ q^{\beta} - q^{2 \alpha - \beta} \right]}{(1 - q^{\alpha - \beta})(1-q^{s+k+2\beta+2j})(1-q^{s+k+2\alpha+2j})}\, \notag \\
& \hspace*{1cm} \quad+ 2 \left( 1-q^2 \right)^{s+k} \sum_{r = 0}^{1} \,\sum_{\alpha = 0}^{\infty} \,\sum_{j=0}^{\infty}  \, \tfrac{\Gamma(s+k+j)}{j! \, \Gamma(s+k)} \, \tfrac{\wt{P_{h_k}}(\alpha,\alpha,r) \, q^{ \alpha + 2 j}}{(1-q^{s+k+2\alpha+2j})^2} \, .\label{zDb}
\end{align}

The structure of the above formula suggests that the poles (of at most second order) of $\zD^{P,k}(s)$ are located at $s \in -\N_{\geq k} + i\tfrac{2 \pi }{\log q} \Z$ with possible cancellations due to vanishing of some of the coefficients. Since $k \geq 0$ we conclude that the set of poles of $\zD^{P_0}$ is included in 
$$
\Sd \vc -\N + i\tfrac{2 \pi }{\log q} \Z.
$$

To prove that \eqref{zDb} gives indeed a maximal analytic extension of $\zD^{P_0}$ and that it has no poles outside of $\Sd$ we need the following lemma.

\begin{Lm}
For any $r \in \{0,1\}$ and $k \in \{0, \ldots n\}$, the following series of \eqref{zDb} are locally uniformly convergent as functions of $s$ on $\C \setminus \Sd$:
\begin{align}
& \sump{\alpha,\beta = 0}{\infty}\, \,\sum_{j=0}^{\infty} \, \tfrac{\Gamma(s+k+j)}{j! \, \Gamma(s+k)} \,  \wt{P_{h_k}}(\alpha,\beta,r) \, q^{2 j + \beta} \,\big[(1 - q^{\alpha - \beta})(1-q^{s+k+2\beta+2j})(1-q^{s+k+2\alpha+2j}) \big]^{-1},\label{SeriesA} \\
& \sump{\alpha,\beta = 0}{\infty} \,\,\sum_{j=0}^{\infty} \, \tfrac{\Gamma(s+k+j)}{j! \, \Gamma(s+k)} \,  \wt{P_{h_k}}(\alpha,\beta,r) \, q^{2 j + 2\alpha - \beta} \big[(1 - q^{\alpha - \beta})(1-q^{s+k+2\beta+2j})(1-q^{s+k+2\alpha+2j})\big]^{-1} ,\label{SeriesB} \\
& \sum_{\alpha = 0}^{\infty} \,\,\sum_{j=0}^{\infty}  \, \tfrac{\Gamma(s+k+j)}{j! \, \Gamma(s+k)} \, \wt{P_{h_k}}(\alpha,\alpha,r) \, q^{ \alpha + 2 j} \,(1-q^{s+k+2\alpha+2j})^{-2} \, . \label{SeriesC}
\end{align}
\end{Lm}

\begin{proof}
The formula (\ref{pf}) for $m=\pm l=\pm \oh$ guarantees the convergence of the following numeric series for $0<q<1$
\begin{align} 
\label{pfconv}
\sum_{\alpha,\beta = 0}^{\infty} \wt{P_{h_k}}(\alpha,\beta,r)
\,q^{\alpha}, \qquad \sum_{\alpha,\beta = 0}^{\infty} \wt{P_{h_k}}(\alpha,\beta,r) 
\,q^{\beta}.
\end{align}
Moreover, the following series of functions of the argument $s$
\begin{align}
\label{Jconv}
\sum_{j = 0}^{\infty} \tfrac{\Gamma(s+k+j)}{j! \, \Gamma(s+k)} \, q^j
\end{align}
is locally uniformly convergent for any $k \in \N$ and any $0 < q < 1$.

In the proof of convergence of the series (\ref{SeriesA}--\ref{SeriesC}) we use the comparative criterion. Recall that, if $\sum_{i \in \N^T} a_i(s)$ is locally uniformly convergent for some $T \in \N_{+}$, then $\sum_{i \in \N^T} a_i(s) \,b_i(s)$ is so when the sequence $\{|b_i(s)|\}_{i \in \N^T}$ is locally uniformly bounded. 
\\
In the series (\ref{SeriesA}--\ref{SeriesB}) we have a multiindex $i = \{\alpha,\beta,j\}$, so let us set
\begin{align*}
a_i^A(s) & = \wt{P_{h_k}}(\alpha,\beta,r) \,q^{\alpha} \, q^j \, \tfrac{\Gamma(s+k+j)}{j! \, \Gamma(s+k)}\,, \\
a_i^B(s) & = \wt{P_{h_k}}(\alpha,\beta,r) \,q^{\beta} \, q^j \, \tfrac{\Gamma(s+k+j)}{j! \, \Gamma(s+k)}\,,\\
b_i^A(s) & =  q^{j}\,\big[(1 - q^{\alpha - \beta})(1-q^{s+k+2\beta+2j})(1-q^{s+k+2\alpha+2j})\big]^{-1}\,, \\
b_i^B(s) & = q^{j + \alpha - \beta}\,\big[(1 - q^{\alpha - \beta})(1-q^{s+k+2\beta+2j})(1-q^{s+k+2\alpha+2j})\big]^{-1}\,.
\end{align*}
Similarly, for (\ref{SeriesC}) (with multiindex $i = \{\alpha,j\}$) we set
\begin{align*}
a_i^C(s) & = \wt{P_{h_k}}(\alpha,\alpha,r) \, q^{\alpha} \, q^j \, \tfrac{\Gamma(s+k+j)}{j! \, \Gamma(s+k)}, \\
b_i^C(s) & =  q^{j}\,(1-q^{s+k+2\alpha+2j})^{-2}.
\end{align*}

Due to the convergence of series \eqref{pfconv} and local uniform convergence of \eqref{Jconv}, 
the series $\sum_{i \in \N^T} a_i^L(s)$ are locally uniformly convergent in each of the three cases ($L = A,B,C$).
\\
Hence, to show the local uniform convergence of the series (\ref{SeriesA}--\ref{SeriesC}), it is sufficient to check that for any $s_0 \in \C \setminus \Sd$, there exist a neighborhood $U \ni s_0$, $U \cap \Sd = \emptyset$ and constants $M = M(U) > 0$, $N = N(U) \in \N$ such that for any $i > N$ (which means $\alpha,\beta,j > N$) we have $|b_i^A(s)|, \, |b_i^B(s)|, \, |b_i^C(s)| \leq M$ for any $s \in U$.
\\
Let us consider first the pointwise limit $\alpha,\beta,j \to \infty$ with $\alpha -\beta = \text{const.} = a \in \Z^*$ (recall that we need $\alpha \neq \beta$ in series (\ref{SeriesA}) and (\ref{SeriesB})) of $b_i^L(s_0)$ with $L = A,B$. It is straightforward that both limits exist for any $a \in \Z^*$, $s_0 \notin \Sd$, $0<q<1$ and are both equal to $0$. Since  $b_i^L$ are analytic near any $s_0 \notin \Sd$, we can always find a suitable neighborhood $U$ and a universal upper bound $M(U) = \sup_{s \in U} \{ M(s) \, \big| \,\forall i > N(s),\, |b_i^L(s)| < M(s) \}$ to guarantee the local uniform boundedness. The same argument shows that $\{|b_i^C(s)|\}_{i \in \N^2}$ is locally uniformly bounded.
\\
Consequently, the series (\ref{SeriesA}--\ref{SeriesC}) are locally uniformly convergent for any $0<q<1$. 
\end{proof}

\begin{proof}[End of proof of Proposition \ref{DimSp}]
To finish the proof we use Weierstrass' theorem which states that if a sequence of analytic functions $\{f_n\}$ converges locally uniformly to a function $f$, then $f$ is also analytic. In the case of locally uniformly convergent series, the sequence is defined as usual in terms of partial sums.

Finally, recall that $\zD^P(s) = \zD^{P_0}(s-n+p)$, so for any $r \in \N$ the poles of $\zD^P$ would form a subset of $-\N_{\geq r} + i\tfrac{2 \pi}{\log q} \Z$, hence in particular of $\Sd$.

We have thus constructed an analytic extension of the function $\zD^P$ for any $P \in \B$ and have shown that the set of poles of any of those functions form a subset of $\Sd$. Moreover, the poles are of at most second order. 
\end{proof}

\begin{Cor}
The dimension spectrum (see Definition \ref{newdimspec}) of the $\Uq$-equivariant Podle\'s sphere is
\begin{align}
\label{Sd}
\Sd(S_q^2) = - \N + i\tfrac{2 \pi}{\log q}\, \Z
\end{align}
and the poles are at most of the second order. 
\end{Cor}

\begin{proof}
The above lemma shows that $\Sd(S_q^2) \subset - \N + i\tfrac{2 \pi}{\log q}\, \Z$. To prove the equality it is sufficient to note that for any $z = -k + i\tfrac{2 \pi}{\log q}\, a$ with some $k \in \N$ and $a \in \Z$, the zeta function $\zD^P$ associated with $P = |\DD|^{-k}$ will have a second order pole at $z$.
\end{proof}

The dimensions spectrum of the standard Podle\'s sphere turns out to be quite different than that for $SU_q(2)$ 
\cite{ConnesSU2} and this is rather unexpected since $S_q^2$ is a quantum homogeneous space of $SU_q(2)$ 
\cite{Podles}. First pathology comes from the fact that the highest real number in the dimension spectrum is $0$ 
hence, $S_q^2$ is 0-summable, see \eqref{traceable}. This drawback, already known in the literature \cite{KrahmerWagner,NeshTuset} can be cured by the introduction of an auxiliary twist \cite{KrahmerWagner}. A second unexpected feature is the existence in the dimension spectrum of poles of second order. Such a behaviour can appear in singular manifolds \cite{Lescure}. This is serious difficulty from the point of view of noncommutative geometry, since many formulas concerning the spectral action \cite{ConnesSA,ConnesInner} have been proved under the assumption of simplicity of the dimension spectrum (note however that there are explicit formulas for the local index without this assumption \cite{ConnesMoscovici}).
It turns out that this trouble also disappears when the twist is employed. Finally, the dimension spectrum of $S_q^2$ contains complex numbers -- a feature already seen in fractals \cite{ConnesMarcolli}. 
Those persist even if the twist is introduced \cite{NeshTuset}.

\section{Heat kernel expansion}\label{action}

In this section we shall provide an \textit{exact} formula for the trace of the heat operator $e^{-t\, |\DD|}$, which is valid \textit{for any $t>0$}. We shall use is the inverse Mellin transform \cite{Flajolet}.

Let us recall the definition of the Mellin transform (and its inverse) of a function $g$:
\begin{align*}
& M[g](s)\vc \int_0^\infty t^{s-1}\,g(t) \,dt,\\
& M^{-1}[g](t)\vc\tfrac{1}{2\pi i}\int_{c-i\infty}^{c+i\infty}t^{-s}\,g(s)\,ds,
\end{align*}
where $s\in \C$ and $t>0$. This needs convergence in some (largest open) infinite vertical strip $(a,b)\subset \C$ with $a,b\in \R \cup \{\pm \infty\}$ for the first integral and $c$  chosen in this strip ($a<c<b$) for the second integral.

To justify the use of Mellin transform we need the following technical lemma.

\begin{Lm}
\label{HeatLemma}
For any $\varepsilon > 0$, the function $t \in(0,\infty) \mapsto t^{\varepsilon} \Tr e^{-t \,|\DD|}$ is continuous and bounded.
\end{Lm}
\begin{proof}
We essentially follow \cite[Lemma 3.4]{NeshTuset}. 
\\
The definition \eqref{Dirac} yields $\Tr e^{-t \,|\DD|} = 4 \sum_{n=1}^{\infty} n\, e^{-t |\omega| [n]_q}$. Thus $t \mapsto t^{\varepsilon} \Tr e^{-t\, |\DD|}$ is continuous and vanishing at infinity for any positive $\varepsilon$ (compare with Lemma \ref{t->infinity}), so we focus on showing boundedness near $t=0$. Let us chose $m \in \N_+$ such that $q^m \leq t \leq q^{m-1}$ and note that $|\omega| [n]_q \geq (q^{-n} - 1) u$, with $u \vc  \tfrac{|w|\,q}{1 - q^2} > 0$. Hence,
\begin{align*}
\sum_{n=1}^{\infty} t^{\varepsilon} \,n \,e^{-t |\omega| [n]_q} & \leq \sum_{n=1}^{\infty} q^{(m-1)\varepsilon} n e^{- (q^{-(n-m)} - q^{m}) u}  \leq q^{-\varepsilon} \sum_{k=-\infty}^{\infty} q^{m\varepsilon} (k+m) e^{- (q^{-k} - q^{m}) u} \\
& \leq q^{-\varepsilon}  e^{u} \big( \sum_{k=-\infty}^{\infty} k e^{- q^{-k} u} + \tfrac{1}{\varepsilon \, e \, \log q^{-1}} \sum_{k=-\infty}^{\infty} e^{- q^{-k} u} \big) < \infty
\end{align*}
where in the second inequality we performed the change of variables $k = n-m$ and took into account that $m$ may be arbitrarily large and in the third, we used the facts $q^m \leq 1$ and $x \,e^{-x} \leq e^{-1}$ for $x>0$. Since $u >0$, the last two series are convergent since and we get an estimate which is independent of $m$. So, using $m \to \infty$, the desired function is uniformly bounded near $t=0$.
\end{proof}

Lemma \ref{HeatLemma} assures that the Mellin transform of $t \mapsto \Tr e^{-t \,|\DD|}$ exists with the fundamental strip $(1,\infty)$. The next lemma proves that the conditions of the Mellin inversion theorem  \cite[Theorem 2]{Flajolet} are met.

\begin{Lm} \label{MellinLemma1}
For any $t>0$ and any $c>0$,
\begin{align}
\Tr e^{-t\, |\DD|} = \tfrac{1}{2 \pi i} \int_{c - i \infty}^{c + i \infty} ds \, t^{-s}\, \Gamma(s)\,\zD(s).  \label{invMellin}
\end{align}
\end{Lm}

\begin{proof}
Lemma \ref{HeatLemma} shows, that $t \mapsto \Tr e^{-t \,|\DD|}$ is indeed continuous and integrable. The analysis we shall perform below shows that the fundamental strip may actually be extended to $(0,\infty)$ since the integral $\int_{c - i \infty}^{c + i \infty} ds \, t^{-s}\, \Gamma(s)\,\zD(s)$ is convergent for any $c >0$.

With
\begin{align*}
g(t) & \vc \Tr e^{-t \,|\DD|} = 4 \sum_{n=1}^{\infty} n \,e^{-t |\omega| [n]_q},
\end{align*}
we have
\begin{align*}
M[g](s)= 4 \sum_{n=1}^{\infty} n \, \int_0^\infty t^{s-1} e^{-t |\omega| [n]_q}\, d t =  4 \Gamma(s) \,|\omega|^{-s} \sum_{n=1}^{\infty} n \, [n]_q^{-s}\, d t = \Gamma(s)\, \zD(s)
\end{align*}
converges for any $s\in \C$ with $\Re(s)>0$,  so $M[g](c+iy)$ converges for any $c>0$. \\
Thus, applying \cite[Theorem 2]{Flajolet}, 
\begin{align*}
g(t)=M^{-1}[M[g]](t)=\tfrac{1}{2\pi i}\int_{c-i\infty}^{c+i\infty}t^{-s}\, M[g](s)\, ds=\tfrac{1}{2\pi i} \int_{c-i\infty}^{c+i\infty}f_t(s)\, ds,
\end{align*} for all $t>0$ where $f_t(s) \vc \,t^{-s}\, \Gamma(s)\, \zD(s)$.
The infinite number of poles on vertical lines at $-\N$ do not pose a problem, because of Corollary 1 of Theorem 2 in \cite{Flajolet} (compare also with Example 12 therein).
\end{proof}

Guided by the Example 12 of \cite{Flajolet}, we shall now prove that the right hand side of \eqref{invMellin} is in fact a sum of residues of the integrand:

\begin{prop} \label{MellinLemma2}
For any $t>0$,
\begin{align*}
\Tr e^{-t \,|\DD|} = \sum_{s_0 \in (\Sd_1 \cup -\N)} \,\underset{s = s_0}{\Res}\, \, t^{-s} \,\zD(s) \,\Gamma(s).
\end{align*}
\end{prop}

\begin{proof}
Let us close the integration contour of the right hand side of \eqref{invMellin} as shown on the  

\begin{tikzpicture}[>=angle 45,line width=0.9pt,x=0.8cm,y=1.3cm,
		point spectre/.style={draw=black,cross out,line width=0.8pt,inner sep=1.5pt,outer sep=0pt},
		decoration={markings,mark=at position 0.9 with {\arrowreversed[black, line width=0.5pt,scale=2]{angle 
		60}}}
		]
\draw[->] (-8,0) -- (2.5,0);
\draw[->] (0,-3.2) -- (0,3.5);
\node[] at (3.2,0) {$\Re(s)$};
\node[] at (0,3.7) {$i\Im(s)$};
\node[] at (0.26,0.2) {$0$};
\draw[line width=2pt] (1,-2.25) --(1,1); \draw[line width=2pt,->] (1,1)--(1,1.2); \path[draw,line width=2pt] (1,1)--(1,2.25) -- (-5.25,2.25) -- (-5.25,-2.25) -- (1.04,-2.25);
{\foreach \x in {0,...,15}
{\foreach \y in {-6,...,6}
           \node[point spectre] at (-\x/2,\y/2) {};
   } 
   }
\node[fill=white,anchor=south west,outer sep=2pt] at (1.2,2.25) {$c+i \wt{R}$};
\node[fill=white,anchor=south west,outer sep=2pt] at (1.2,-2.75) {$c-i \wt{R}$};
\node[fill=white,anchor=south east,outer sep=2pt] at (-5.25,2.25) {$-R+i \wt{R}$};
\node[fill=white,anchor=south east,outer sep=2pt] at (-5.25,-2.75) {$-R-i \wt{R}$};
\node[fill=white,anchor=south east,outer sep=3pt] at (-1.63,1.55) {\large$\boldsymbol{I_{H^+}}$};
\node[fill=white,anchor=south east,outer sep=2.pt] at (-1.7,-2.255) {\large$\boldsymbol{I_{H^-}}$};
\node[fill=white,anchor=south east,outer sep=3pt] at (-3.95,.05) {\large$\boldsymbol{I_{V}}$};
\node[] at (1.3,0.17) {$c$};
\node[text width=6cm] at (8,1) {figure where we have set 
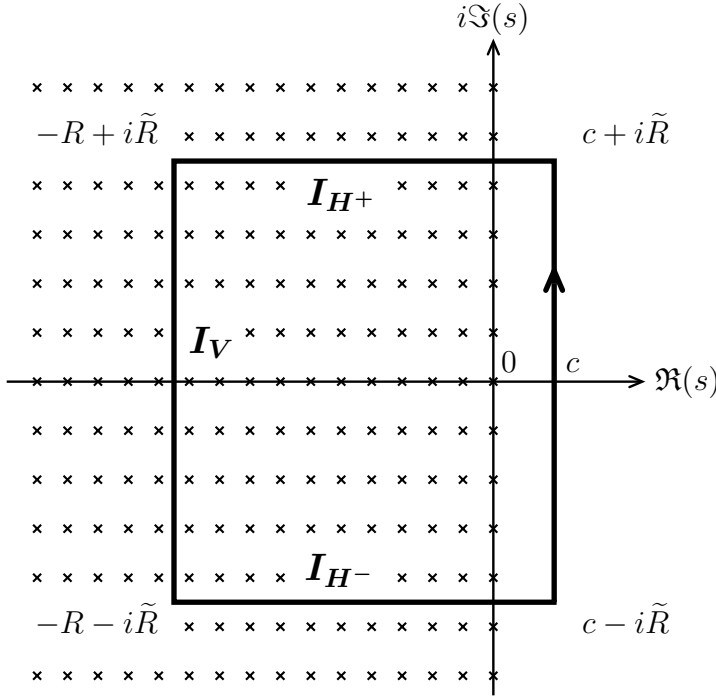
\begin{align*}
&\wt{R} \vc \eta R,\\ 
& \eta \vc -\tfrac{2\pi}{\log q} > 0
\end{align*}
and the crosses are the poles of the function  $$f_t(s) = t^{-s} \,\Gamma(s)\,\zD(s).$$
  \quad We shall increase the value of 
$R$ in a discrete way $R = N + \oh$, with $N \in \N$ for the contour to avoid the poles.
};
\end{tikzpicture}

\quad Using Lemma \ref{MellinLemma1} and the Residue Theorem, we have
\begin{align*}
\Tr e^{-t \,|\DD|} & = \lim_{R \to \infty} \, \left( \sum_{\a \in S_R} \underset{s = \a}{\Res} \, f_t(s) + I_{H^+}(R) + I_{H^-}(R) + I_V(R) \right) \\
& = \sum_{s_0 \in (\Sd \cup -\N)} \,\underset{s = s_0}{\Res}\, \, f_t(s) + \lim_{R \to \infty} \left( I_{H^+}(R) + I_{H^-}(R) + I_V(R) \right),
\end{align*}
where $S_R$ denotes the set of poles of $f$ lying inside the rectangular contour of integration. The second equality is sound whenever both the infinite series of residues and the contour integrals converge. Note, that the the first equality is true by Lemma \ref{MellinLemma1}, so if the contribution of the contour integrals is finite in the $R \to \infty$ limit, then the infinite sum of residues must be finite as well, so the series of residues converges. In particular, to demonstrate the statement of the lemma we shall show that
\begin{align}
\label{lim}
\lim_{R \to \infty} |I_{H^{\pm}}(R)| = \lim_{R \to \infty} |I_V(R)| = 0.
\end{align}

We will need a control of the $\Gamma$-function on vertical and horizontal lines: using the expression for the logarithm of the Gamma function \cite[eq. (2.1.1)]{Paris} we have
\begin{align}
\label{Gamma}
\Gamma(z) =(2\pi)^{1/2} \,e^{-z}\,z^{z-1/2} e^{\Omega(z)}, \text{ for } |\arg(z)| < \pi.
\end{align}
Writing down the real and imaginary parts explicitly, we obtain
\begin{align}
|\Gamma(x+ i y)| & = (2\pi)^{1/2} \,e^{-x}\, \left\vert (x+ i y)^{x-1/2 + i y} \right\vert \, \left\vert e^{\Omega(x+iy)} \right\vert \notag \\
& = (2\pi)^{1/2} \,e^{-x- y \arg(x+iy)}\, (x^2 + y^2)^{x/2 - 1/4} \, \left\vert  e^{\Omega(x+iy)} \right\vert \notag \\
& = (2\pi)^{1/2} \,e^{-x- |y| \,|\arg(x+iy)|}\, (x^2 + y^2)^{x/2 - 1/4} \, \left\vert   e^{\Omega(x+iy)}\right\vert. \label{Gamma1}
\end{align}
Note that since the remainder term $\Omega$ can be written as \cite[eq. (2.1.2)]{Paris}
\begin{align*}
\Omega(z) = 2 \int_{0}^{\infty} \tfrac{\arctan(t/z)}{e^{2 \pi t}-1} \, dt, \text{ for } \Re(z) > 0,
\end{align*}
we have $\left\vert e^{\Omega(z)} \right\vert  = e^{\Re \,( \Omega(z)) } \geq 1, \text{ for } \Re(z) > 0$. So, using \eqref{Gamma1}, a lower bound on $|\Gamma(z)|$ can be given: $|\Gamma(x+ i y)|  \geq (2\pi)^{1/2} e^{-x- |y| |\arg(x+iy)|}(x^2 + y^2)^{x/2 - 1/4}, \text{ for } x>0$. Hence
\begin{align} \label{GammaUp}
|\Gamma(x+ i y)|^{-1} \leq (2\pi)^{-1/2} \,e^{x+ |y| \,|\arg(x+iy)|}\, (x^2 + y^2)^{-x/2 + 1/4}, \text{ for } x>0.
\end{align}

Let us first start with the estimation of $|I_{H^{\pm}}(R)|$ for large positive $R$. 
\begin{align*}
& |I_{H^{\pm}}(R)| = \Big| \int_{-R}^c f_t(x \pm i \wt{R})\,dx \Big| \leq 4 \int_{-R}^c \, dx \, (\tfrac{1-q^2}{\vert w \vert \,t})^{x} \, \sum_{n=0}^{\infty} \, \tfrac{|\Gamma (x+n\pm i \wt{R})|}{n!} \, \tfrac{q^{2n}}{ {\vert 1-q^{x\pm i \wt{R}+2n}\vert^2}^{\,}}.
\end{align*}

We have
\begin{align*}
\big\vert 1-q^{x\pm i \wt{R}+2n}\big\vert^2 = \big\vert 1 - q^{x +2n \pm 2 \pi i (N + 1/2)/\log q} \big\vert^2 = (1 + q^{x+2n})^2 \geq 1\,,
\end{align*}
so
\begin{align*}
|I_{H^{\pm}}(R)| & \leq 4 \int_{-R}^c \, dx \, (\tfrac{1-q^2}{\vert w \vert t})^{x} \, \sum_{n=0}^{\infty} \, \tfrac{q^{2n}}{n!} \, |\Gamma (x+n\pm i \wt{R})|.
\end{align*}
Application of formula \eqref{Gamma1} gives
\begin{align*}
&|I_{H^{\pm}}(R)|  \leq 4 \sqrt{2 \pi} \int_{-R}^c \, dx \, (\tfrac{1-q^2}{\vert w \vert\, t \,e})^{x} \,\\
&\hspace{4cm}  \sum_{n=0}^{\infty} \, \tfrac{e^{-n} \, q^{2n}}{n!} \, \left( (x+n)^2 + \wt{R}^2 \right)^{(x+n)/2 - 1/4} \, e^{-\wt{R} \,|\arg(x+n \pm i\wt{R})|} \, \left\vert   e^{\Omega(x+n+i\wt{R})}\right\vert .
\end{align*}

The remainder term can be given the following estimate \cite[eq. (2.1.5) with $n=1$ and eq. (2.1.6)]{Paris} where we assume $\wt{R}>1$:
\begin{align}
\left\vert  \Omega(x + n \pm i\wt{R}) \right\vert &\leq \tfrac{1}{12 \cos^2(\oh \, \arg(x +n \pm i \wt{R}))\, |x +n \pm i\wt{R}|^{^{\,}}} \leq \tfrac{1}{12 \sin^2(\oh \, \arctan \eta\, ( (x+n)^2+\wt{R}^2)^{^{1/2}}} \nonumber\\
& \leq \tfrac{c(q)}{(x^2+1)^{1/2}}\,,\label{ErrorEst}
\end{align}
with $c(q) \vc (12 \sin^2(\oh \, \arctan \eta))^{-1} > 0$; since $-R \leq -R+n \leq x+n$, the second inequality comes from \begin{align*}
\cos^2\big(\oh \, \arg(x + n \pm i \wt{R})\big) &\geq \cos^2 \big(\oh \, \arg(-R \pm i \wt{R})\big) = \cos^2\big(\oh (\pm \pi \mp \arctan \eta)\big) \\
&= \sin^2(\oh \arctan \eta).
\end{align*}

Without the loss of generality, we can assume $c = 1/2$ so that $x/2 - 1/4 \leq 0$. \\
Since $(x+n)^2 + \wt{R}^2 \geq \wt{R}^2$  and with the abbreviation $v \vc \tfrac{1-q^2}{\vert w \vert \,t\, e}$ ,
\begin{align*}
|I_{H^{\pm}}(R)| & \leq 4 \sqrt{2 \pi} \int_{-R}^{1/2} \, dx \, v^{x} \, \wt{R}^{x - 1/2} \, e^{c(q)(x^2+1)^{-1/2}} \, S(x,\wt{R}) \\
& \leq 4 \sqrt{2 \pi} \, e^{c(q)} \, \wt{R}^{- 1/2} \, \int_{-R}^{1/2} \, dx \, (v \wt{R})^{x} \, S(x,\wt{R})
\end{align*}
where
\begin{align*}
S(x,\wt{R}) &\vc \sum_{n=0}^{\infty} \, \tfrac{e^{-n} \, q^{2n}}{n!} \, \left( (x+n)^2 + \wt{R}^2 \right)^{n/2} \, e^{-\wt{R} \,|\arg(x+n \pm i\wt{R})|} \\
& \,\leq\sum_{n=0}^{\infty} \, \tfrac{e^{-n} \, q^{2n}}{n!} \, \left( (x+n)^2 + \wt{R}^2 \right)^{n/2} \, e^{-\wt{R}\, |\arg(\oh+n \pm i\wt{R})|},
\end{align*}
and in the second line we have used the fact that for any $n \geq 0 $ and $-R \leq x \leq \oh$ we get $|\arg(x+n\pm i \wt{R})| \geq |\arg(\oh+n\pm i \wt{R})|$.

Let us now divide the series of $S$ on $n$ into two pieces: $n \leq R-\oh$ and $n \geq R + \oh$ (recall that we have assumed that $R = N + \oh$, with $N \in \N$). For the first of those we have
\begin{align*}
& \sum_{n=0}^{R-1/2} \, \tfrac{e^{-n} \, q^{2n}}{n!} \, \left( (x+n)^2 + \wt{R}^2 \right)^{n/2} \, e^{-\wt{R} |\arg(\oh +n \pm i\wt{R})|} \\
& \hspace*{3cm} \leq
e^{-\wt{R} |\arg(R \pm i\wt{R})|} \, \sum_{n=0}^{R-1/2} \, \tfrac{e^{-n} \, q^{2n}}{n!} \, \left( (x+n)^2 + \wt{R}^2 \right)^{n/2} \\
& \hspace*{3cm} \leq
e^{-\wt{R} \arctan (\eta)} \, \sum_{n=0}^{R-1/2} \, \tfrac{e^{-n} \, q^{2n}}{n!} \, \left( (R+n)^2 + \wt{R}^2 \right)^{n/2} \\
& \hspace*{3cm} \leq
e^{-\wt{R} \arctan (\eta)} \, \sum_{n=0}^{R-1/2} \, \tfrac{e^{-n} \, q^{2n}}{n!} \, \left( (2R)^2 + \wt{R}^2 \right)^{n/2} \\
& \hspace*{3cm} \leq
e^{-\wt{R} \arctan (\eta)} \, \sum_{n=0}^{\infty} \, \tfrac{e^{-n} \, q^{2n}}{n!} \, \left( (4 + \eta^2) R^2 \right)^{n/2} \\
& \hspace*{3cm} = \exp \left( -R \big( \eta \arctan (\eta) - e^{-1} q^2 \sqrt{4 + \eta^2} \big) \right) =  e^{ -\eta F(\eta) R},
\end{align*}
with $F(\eta) \vc \arctan (\eta) - e^{-1-4\pi/\eta} \sqrt{1 + \tfrac{4}{\eta^2}}$, since $\eta = -\tfrac{2\pi}{\log q}$. \\
Note that $F(\eta) > 0$ when $\eta > 0$ (this can be checked by using $\text{Reduce}[F[\eta^{-1}] > 0,\eta]$ in Mathematica \cite{Mathematica}), so $F\big(\eta(q)\big) >0$, for $0< q< 1$.

To proceed with the remainder of the series, we remark that, since $x\in [-R,1/2]$, for $R<n\in \N$ we get: $0<-R+n \leq x+n\leq 1/2+n<n+1$, so:
\begin{align} \label{fact1}
(x+n)^2 + \wt{R}^2 \leq (n+1)^2 +  \wt{R}^2,\quad \text{for } n > R, \, x \in [-R,\oh],
\end{align}
Now, using the fact \eqref{fact1} and the formula \eqref{GammaUp} for $n! = \Gamma(n+1)$ we obtain
\begin{align*}
& \sum_{n=R+1/2}^{\infty} \, \tfrac{e^{-n} \, q^{2n}}{n!} \, \left( (x+n)^2 + \wt{R}^2 \right)^{n/2} \, e^{-\wt{R} \,|\arg(\oh +n \pm i\wt{R})|} \\
& \hspace*{2.5cm} \leq
\sum_{n=R+1/2}^{\infty} \, \tfrac{e^{-n} \, q^{2n}}{n!} \, \left( (n+1)^2 +  \wt{R}^2 \right)^{n/2} \, e^{-\wt{R}\, |\arg(\oh +n \pm i\wt{R})|} \\
& \hspace*{2.5cm} \leq
\sum_{n=R+1/2}^{\infty} \, \tfrac{e^{-n} \, q^{2n}}{\Gamma(n+1)} \, \left( (n+1)^2 +  \wt{R}^2 \right)^{(n+1)/2} \, e^{-\wt{R}\, |\arg(1 +n \pm i\wt{R})|} \\
& \hspace*{2.5cm} \leq
\tfrac{e}{\sqrt{2 \pi}} \, \sum_{n=R+1/2}^{\infty} \, q^{2n} \,(n+1)^{-(n+1)+1/2} \, \left( (n+1)^2 +  \wt{R}^2 \right)^{(n+1)/2} \, e^{-\wt{R}\, |\arg(1 +n \pm i\wt{R})|} \\
& \hspace*{2.5cm} \leq
\tfrac{e}{\sqrt{2 \pi}} \, \sum_{n=R+1/2}^{\infty} \, \sqrt{n+1} \, q^{2n} \, \left( 1 +  \left(\wt{R}/(n+1) \right)^2 \right)^{(n+1)/2} \, e^{-\wt{R} \arctan (\wt{R} / (n+1) )} \\
& \hspace*{2.5cm} =
\tfrac{e}{\sqrt{2 \pi}} \, \sum_{n=R+1/2}^{\infty} \, \sqrt{n+1} \, q^{2n} \, e^{ \tfrac{n+1}{2} \log \left( 1 +  ( \tfrac{\wt{R}}{n+1})^2 \right) -\wt{R} \arctan (\tfrac{\wt{R}}{n+1})} \\
& \hspace*{2.5cm} =
\tfrac{e}{\sqrt{2 \pi}} \, \sum_{n=R+1/2}^{\infty} \, \sqrt{n+1} \, q^{2n} \, e^{- \wt{R} \,\, G\left(\wt{R}/(n+1) \right)}
\end{align*}
with $G(y) := \arctan y - \tfrac{1}{ y} \log(1+y^2)$. The function $G$ turns out to be positive for any $y>0$ (with $\lim_{y \to \infty} G(y) \to \pi/2$). So we can write
\begin{align*}
\sum_{n=R+1/2}^{\infty} \, \sqrt{n+1} \, q^{2n} \, e^{ - \wt{R} \,\, G( \tfrac{\wt{R}}{n+1})} & \leq \sum_{n=R+1/2}^{\infty} (n+1) \, q^{2n} =q^{2R+1}(\tfrac{1}{1-q^2}\,R+\tfrac{3-q^2}{2(1-q^2)^2})  \leq \wt{c}(q) \,R\, q^{2R},
\end{align*}
with a constant $\wt{c}$ depending only on $q$.

Finally, putting the two parts of the series together, we obtain 
\begin{align*}
|I_{H^{\pm}}(R)|  \leq 4 \sqrt{2 \pi} e^{c(q)} \wt{R}^{-1/2} \, \int_{-R}^{1/2} \, dx \, (v \wt{R})^{x} \, 
\left( e^{-\eta\, F(\eta) \,R} + \wt{c}(q)\, R \,q^{2R} \right),
\end{align*}
and
\begin{align*}
\lim_{R\to \infty} |I_{H^{\pm}}(R)| &\leq \lim_{R\to \infty} 4 \sqrt{2 \pi} e^{c(q)} \wt{R}^{-1/2} \, 
\left( e^{-\eta\, F(\eta) \,R} + \wt{c}(q)\, R \,q^{2R} \right) \, \tfrac{1}{\log (v \wt{R})} \, \left(  (v \wt{R})^{1/2}-(v \wt{R})^{-R}  \right) \\
 &\leq \lim_{R\to \infty}\text{const.} \,\tfrac{R}{\log R} \,e^{-H(q) \,R} = 0,
\end{align*}
where $H(q) > 0$ for any $0< q <1$. \\
Thus the contribution of the integrals along the horizontal lines vanish. 

Let us now focus on the vertical line. 
 We have
\begin{align*}
|I_{V}(R)| & \leq \int_{-\wt{R}}^{\wt{R}} dy \, |f_t(-R + i y) |  \leq 4 t^{R} \, \int_{-\wt{R}}^{\wt{R}} \,dy\, |\Gamma(-R + i y)| \, |\zD(-R + iy)|.
\end{align*}
Since $R = N + \oh$, with $N \in \N$ and for any $0<q<1$
\begin{align*}
\min \{ (1-q^{3/2})^2, \, (1-q^{1/2})^2, (1-q^{-1/2})^2, (1-q^{-3/2})^2 \} = (1-\sqrt{q})^2,
\end{align*}
we have using \eqref{zDestim}
\begin{align*}
|\zD(-R + iy)| \leq 4 (1-\sqrt{q})^{-2} \, (\tfrac{1-q^2}{\vert w \vert})^{-R} \, (1-q^2)^{-\vert R + i y\vert}.
\end{align*}
As  $|\Gamma(x+i y)| \leq |\Gamma(x)|$ \cite[Ex. 1.4, p. 38]{Olver}, we can estimate
\begin{align*}
|I_{V}(R)| & \leq 4 (1 - \sqrt{q})^{-2} \nu^{R} \, \int_{-\wt{R}}^{\wt{R}} \,dy \, |\Gamma(-R + i y)| \, (1-q^2)^{-\vert R + i y\vert} \\
& \leq 4 (1 - \sqrt{q})^{-2} \nu^{R} \, 2\wt{R} \, |\Gamma(-R)| \, (1-q^2)^{-\sqrt{1+\eta^2} R},
\end{align*}
with the abbreviation $\nu \vc \tfrac{\vert w \vert\, t}{1-q^2} > 0$.\\
To estimate $|\Gamma(-R)|$, we shall use the Euler reflection formula \cite[eq. (2.1.20)]{Paris}
\begin{align*}
|\Gamma(-R)| & = \tfrac{1}{|\Gamma(1+R)|} \, \tfrac{\pi}{|\sin (\pi(-R))|} = \tfrac{\pi}{|\Gamma(1+R)|}  \leq \sqrt{\tfrac{\pi}{2}} \, e^{1+R} |1+R|^{-R - 1/2}  \leq \sqrt{\tfrac{\pi}{2}} \,e \, e^{R} R^{-R - 1/2}
\end{align*}
where we have used the fact that $R = N + \oh$ and the formula \eqref{GammaUp}.

Hence, finally we obtain
\begin{align*}
|I_{V}(R)| & \leq 4 \sqrt{2 \pi} e (1 - \sqrt{q})^{-2} \,\nu^{R}\, e^R \, \wt{R} \, (1-q^2)^{-\sqrt{1+\eta^2} R} R^{-R-1/2} \xrightarrow[R \to \infty]{} 0.
\end{align*}
The heavy dumping $R^{-R}$ assures that the contribution from the vertical part of the contour integral will vanish in the limit $R \to \infty$ independently of the values of $q$ and $t$.

We have thus shown that \eqref{lim} holds true and completed the proof of the proposition.
\end{proof}

We have shown that the trace of the heat operator $e^{-t|\DD|}$ is given by a (doubly infinite) sum of residues of the function $ s \mapsto t^{-s} \,\Gamma(s)\, \zD(s)$. In order to present the result in a compact form it will be convenient to introduce some special functions. Let
us take the Bessel functions of first kind,
\begin{align} \label{J}
J_{\nu}(z)\vc \sum_{k=0}^\infty \tfrac{(-1)^k}{k!\,\Gamma(\nu +k+1)} \,(\tfrac{1}{2}z)^{2k+\nu},\quad z, \nu \in \C,
\end{align}
and introduce new special functions, defined for $z, \nu \in \C$ and $n \in \N$ 
\begin{align}\label{DJ}
\wt{J}^{\,(n)}_{\nu} (z) & \vc  \tfrac{\partial^n\,}{\partial \alpha^n} \big[ (\tfrac{1}{2} z)^{-\alpha} J_{\alpha + \nu} (z) \big] \big|_{\alpha = 0}   \; = \sum_{k=0}^{\infty} \tfrac{(-1)^k }{k!} (\tfrac{1}{2} z)^{2k + \nu} \,
\big( \tfrac{\partial^n}{\partial \alpha^n} \,\tfrac{1}{\Gamma(k + \alpha + \nu + 1)}\big)  \vert_{\alpha = 0} \,.
\end{align}

By \eqref{GammaUp} we have
\begin{align*}
\tfrac{1}{\Gamma(k + \nu + 1)} \underset{k \to \infty}{\sim} \; e^k \, |k+\nu+1|^{-k-\nu -1/2},
\end{align*}
for any $k \in \N$, $\nu \in \C$ and $\vert \arg(k+\nu+1) \vert<\pi$. Moreover, we have the following large argument behaviour of polygamma functions
\begin{align*}
\psi(z) = \log z + \OO(z^{-1}), \qquad \psi^{(n)}(z) = (-1)^{n+1} \tfrac{\Gamma (n)}{z^n} + \OO(z^{n-1}).
\end{align*}
We conclude that for large $k$ we have
$$\big( \tfrac{\partial^n}{\partial \alpha^n} \,\tfrac{1}{\Gamma(k + \alpha + \nu + 1)}\big)  \vert_{\alpha = 0} \; \underset{k \to \infty}{\sim} \; \tfrac{(-1)^{n}}{\sqrt{2 \pi}} \,e^k \,|k+\nu+1|^{-k-\nu -1/2} \, \log^n k,$$
hence the power series defining $\wt{J}^{\,(n)}_{\nu}$ have infinite radius of convergence for any $n \in \N$ and any $\nu \in \C$.

Now we are ready to state the main result concerning the trace of heat operator associated with $\DD$.

\begin{theorem}
\label{PropHeat}
For any $t>0$ and $0<q<1$, the trace of the heat operator associated with $\DD$ has the following form 
\begin{align}
\Tr e^{-t \,|\DD|} = \tfrac{1}{\log^2 q} \,\big[ g(t) \, \log^2 t + \big( h_0(t) 
+ \sum_{a \in \Z^*} h_a(t) \big) \, \log t  + c_0(t) + \sum_{a \in \Z^*} c_a(t) \big] + \sum_{n = 1}^{\infty} d_n \,t^{2n} \label{heat1}
\end{align}
where the coefficients are defined by
\begin{align*}
g(t) & \vc 2 J_0 (2 u t)\,, \\
h_0(t) & \vc 4 \log u \, J_0(2ut) + 4 \wt{J}^{\,(1)}_{0} (2ut)\,, \\
c_0(t) & \vc \big( 2 \log^2 u + \tfrac{2}{3} \pi^2 - \tfrac{1}{3} \log^2 q \big) J_0(2ut)+ 4 \log u \, \wt{J}^{\,(1)}_{0} (2ut) 
+ \wt{J}^{\,(2)}_{0} (2ut)\,,\\
h_a(t) & \vc - i4  \pi \,J_{i \,\tilde{a}}(2 u t)\,\tfrac{1}{\sinh( \pi\tilde{a})}\,,\\
c_a(t) & \vc - i4  \pi \,\,\wt{J}_{i \,\tilde{a}}^{(1)}(2 u t) \,\tfrac{1}{\sinh( \pi\tilde{a})} 
+ 4 \pi^2 \,J_{i \,\tilde{a}}(2 u t) \,\tfrac{ \coth(\pi\tilde{a})}{\sinh(\pi \tilde{a})} \,,\\
d_n & \vc 4 \left(\tfrac{u}{q}\right)^{2n} {\sum_{k=0,\,k\neq n}^{2n}} \tfrac{(-1)^k}{k!(2n-k)!} \, \,q^{2k}\,(1-q^{2(k-n)} )^{-2} \, ,
\end{align*}
with 
\vspace{-0.5cm}
\begin{align*}
\wt{a} \vc  \tfrac{2 \pi}{\log q} \, a, \qquad u \vc  \tfrac{|w|\,q}{1 - q^2} > 0\,.
\end{align*}
\end{theorem}

\noindent Before proving the above result let us make few remarks.

\begin{rem}
\label{rem1}
Reappearance of \,$\log t$-terms:\\
{\rm The function $\wt J_0^{(1)}$ can be expressed in terms of standard Bessel functions, namely
\begin{align}
\wt J_0^{(1)}(z) &=\tfrac{\partial\,}{\partial \alpha} \big[ (\tfrac{1}{2} z)^{-\alpha} J_{\alpha} (z) \big] \big|_{\alpha = 0} = -\log(\tfrac{z}{2})\,J_0(z) + \tfrac{\partial\,}{\partial \alpha} \big[ J_{\alpha} (z) \big] \big|_{\alpha = 0}  \nonumber\\
&=-\log(\tfrac{z}{2})\,J_0(z)+\tfrac{\pi}{2} \,Y_0(z) \label{Jtilde}
\end{align}
where $Y_0 \vc \tfrac{2}{\pi}\,\tfrac{\partial\,}{\partial \nu}J_\nu  \vert_{\nu = 0}$ is the Bessel function of the second kind of order 0 \cite[p. 243]{Olver}. 

However, for arbitrary $n \in \N$ and $\nu \in \C$ the expansion of $\wt{J}^{\,(n)}_{\nu}$ in terms of known functions, if possible at all, would only obscure the properties of newly defined functions. The reason is that the formula \eqref{Jtilde} suggests that $\wt J_0^{(1)}(z)$ is singular at $z=0$ due to the presence of $\log(\tfrac{z}{2})$ term. But the behaviour of $Y_0(z)$ near $z = 0^+$ is also logarithmic \cite[p. 243]{Olver} and both singularities precisely cancel out, so that $\wt J_0^{(1)}(z)$ is regular at $z=0$ and
\begin{align*}
\wt J_0^{(1)}(0) = \gamma,
\end{align*}
where $\gamma$ denotes the Euler gamma constant. Similarly we have $$\wt J_0^{(2)}(0) = \gamma^2 - \tfrac{\pi^2}{6}.$$
On the other hand, for purely imaginary $\nu$, the functions $\wt{J}^{\,(n)}_{\nu}$ including $n=0$ - the standard Bessel function, are not regular at $z=0$. Near $z=0$ we have $J_{\nu}(z) = \left(\tfrac{z}{2}\right)^{\nu} \tfrac{1}{\Gamma(\nu+1)} (1 + \OO(z^2))$ \cite{Olver}, so the behaviour at the origin is oscillatory. This behaviour is reflected by the coefficients $\sum_{a \in \Z^*} h_a(t)$ and $\sum_{a \in \Z^*} c_a(t)$ in the formula \eqref{heat1}. However, those oscillations are bounded since $\sum_{a \in \Z^*} h_a(t)$ and $\sum_{a \in \Z^*} c_a(t)$ converge for any $t>0$.

This shows that our decomposition in \eqref{heat1} in terms of the $\wt{J}^{\,(n)}_{\nu}$-functions is convenient since it clearly singles out the different parts of the trace of the heat kernel operator: the singular one is proportional to $\log t$ and $\log^2 t$, the oscillatory one is $\sum_{a \in \Z^*} c_a(t)$ while the remainder is regular at $t=0$. It is also convenient to separate this remainder in $c_0(t)$ and $\sum_{n = 1}^{\infty} d_n t^{2n}$, since they are of different origin as will become clear in the course of the proof. 
Let us also mention that the oscillatory behaviour of the heat kernel for the standard Podle\'s sphere has been detected in \cite[Lemma 3.4]{NeshTuset} in the context of a twisted version of the index theorem.}
 \end{rem} 

\begin{rem}
\label{rem2}
About the reappearance of complex numbers:\\
{\rm Since $\Tr e^{-t \,|\DD|}$ is a real function of $t$, one may be worried about the appearance of complex numbers in 
coefficients $h_a(t)$ and $c_a(t)$. However, the imaginary parts of $h_a(t)$ and $c_a(t)$ are 
antisymmetric in $a$, hence will cancel out in (\ref{heat1}) because of summation over $a \in \Z^*$. Indeed, the 
definition (\ref{DJ}) of functions $\wt{J}^{(n)}_{\nu}$ involve $\Gamma$ functions of a complex argument 
and $\Im (\Gamma(z)) = - \Im (\Gamma(\bar{z}))$ and $\Re (\Gamma(z)) = \Re (\Gamma(\bar{z}))$ for any $z \in \C$. By 
combining this fact with the parity of hyperbolic functions, we get the desired result.}
\end{rem} 

\begin{rem} 
\label{rem3}
All series in \eqref{heat1} are absolutely convergent:
\\
{\rm 
This remark concerns the convergence of the series $\sum_{a \in \Z^*} h_a(t)$, $\sum_{a \in \Z^*} c_a(t)$ and $\sum_{n = 1}^{\infty} d_n t^{2n}$.  As we mentioned at the beginning of the proof of Proposition \ref{MellinLemma2}, a finite (equal to 0 in particular) contribution from the contour integral guarantees that the sum of residues is finite for all $t >0$. It turns out, that those series are in fact absolutely convergent for any fixed $t$ as we shall show below. 

Let us first investigate the behaviour of Bessel functions $J_{i\, \wt{a}}$ for large $\wt{a}$: as $\nu$ tends to $\infty$ in 
the sector $|\arg \nu| < \pi$, we have the following expansion (compare with \cite[p. 374]{Olver})
\begin{align}
\label{Jnu}
J_{\nu}(z) = \tfrac{1}{\Gamma(1+\nu)}(\tfrac{1}{2} z)^{\nu}\, \left( 1 + \OO\left(\tfrac{1}{\nu}\right) \right), 
\quad \text{ for } |\nu| \gg |\tfrac{1}{2} z|^2, z\in \C
\end{align}
which follows from 
\begin{align*}
\tfrac{1}{\Gamma(\nu + k + 1)} = \tfrac{1}{(\nu + k) \ldots (\nu + 1)\Gamma(\nu+1)} 
= \tfrac{\nu^{-k}}{\Gamma(1+\nu)} \left( 1 + \OO\left(\tfrac{1}{\nu}\right) \right), \quad \forall k \in \N,
\end{align*}
and can be used in \eqref{J}. From Euler reflection formula
\begin{align}
\label{Euler}
\Gamma(z) \,\Gamma(1-z) = \tfrac{\pi}{\sin{\pi z}} \quad \text{for any } z \in \C \setminus \Z,
\end{align}
we deduce the following
\begin{align*}
|\Gamma(1 + i \wt{a})|^2 = \Gamma(1 + i \wt{a}) \,\Gamma(1 - i \wt{a}) 
= i \wt{a} \,\Gamma(i \wt{a}) \,\Gamma(1 - i \wt{a}) = \tfrac{\pi \tilde{a} }{\sinh(\pi \tilde{a})}\,.
\end{align*}
So we finally get, assuming $(ut)^2<\tilde{a}$ by \eqref{Jnu}
\begin{align*}
|h_a(t)| = \left| \tfrac{4 \pi \, (ut)^{i \tilde{a}}}{\Gamma(1 + i \tilde{a}) \sinh(\pi \tilde{a}) } 
\right| \left( 1 + \OO(\wt{a}^{-1})\right) 
= \tfrac{4 \pi}{\sqrt{\pi \tilde{a} \sinh(\pi \tilde{a})}} \left( 1 + \OO(\wt{a}^{-1}) \right),
\end{align*}
what guarantees the absolute convergence of $\sum_{a \in \Z^*} h_a(t)$.
\\
We now pass on to the analysis of the coefficients $c_a(t)$. For the digamma-function $\psi$ that appears 
in the $\wt{J}^{(1)}_{i \wt{a}}$ function \eqref{DJ}, the following recurrence relation holds: $\psi(z+1) = \psi(z) + z^{-1}$
thus $\psi(\nu +k+1) = \psi(\nu) + \OO(\nu^{-1})$. As a consequence of the latter and the previous analysis for 
$J_{\nu}$, we get
\begin{align*}
\wt{J}^{(1)}_{\nu}(z) = - \tfrac{1}{\Gamma(\nu+1)} \,(\tfrac{1}{2} z)^{\nu} \,\psi(\nu)\left( 1 + \OO(\nu^{-1}) \right), 
\quad \text{ for } |\nu| \gg |\tfrac{1}{2} z|^2, \, z\in\C.
\end{align*}
The digamma function behaves logarithmically for large $\vert z \vert$: $\psi(z) = \log(z) + \OO(z^{-1})$  
\cite[p. 39]{Olver}.
Thus
\begin{align*}
|\psi(i \wt{a})| = \sqrt{(\log{|\wt{a}|})^2  + \tfrac{\pi^2}{4}} + \OO(\wt{a}^{-1}) 
= \log{|\wt{a}|} + \OO\left(\log^{-1}|\wt{a}|\right),
\end{align*}
and we arrive at the conclusion
\begin{align*}
|c_a(t)| & = \tfrac{4 \pi \log{|\tilde{a}|}}{\sqrt{\pi \tilde{a} \sinh(\pi \tilde{a})}} \big( 1 + \OO (\log^{-1}|\wt{a}|) \big), \quad \text{ for } t<\sqrt{\tilde{a}} u^{-1},
\end{align*}
showing that $\sum_{a \in \Z^*} c_a(t)$ is indeed absolutely convergent.
\\
Finally, let us address the problem of convergence of series $\sum_{n = 1}^{\infty} d_n t^{2n}$. We have
\begin{align*} 
\sum_{n = 1}^{\infty} |d_n \,t^{2n}|  \leq 4 \sum_{n = 1}^{\infty} \left(\tfrac{ut}{q}\right)^{2n} {\sum_{k=0,\,k\neq n}^{2n}} \tfrac{1}{k!(2n-k)!} \, \tfrac{q^{2k}}{\left(1-q^{2(k-n)} \right)^2}.
\end{align*}
Let us split the sum over $k$ into two parts: $0 \leq k \leq n-1$ and $n+1 \leq k \leq 2n$. The first inequality implies that $(1-q^{2(k-n)})^{-2} \leq (1-q^{-2})^{-2}$, so 
\begin{align*}
\sum_{k=0}^{n-1} \tfrac{1}{k!(2n-k)!} \, \tfrac{q^{2k}}{\left(1-q^{2(k-n)} \right)^2} \leq 
\tfrac{1}{(1-q^{-2})^{2}} \sum_{k=0}^{n-1} \tfrac{q^{2k}}{k!(2n-k)!} <\tfrac{1}{(1-q^{-2})^{2}} \sum_{k=0}^{\infty} \tfrac{q^{2k}}{k!(2n-k)!}= \tfrac{1}{(1-q^{-2})^{2}} \tfrac{(1+q^2)^{2n}}{(2n)!}\,.
\end{align*}
Similarly, $n+1 \leq k \leq 2n \: \Rightarrow \: (1-q^{2(k-n)})^{-2} \leq (1-q^{2})^{-2}$, so
\begin{align*}
\sum_{k=n+1}^{2n} \tfrac{1}{k!(2n-k)!} \, \tfrac{q^{2k}}{\left(1-q^{2(k-n)} \right)^2} \leq 
\tfrac{1}{(1-q^{2})^{2}} \sum_{k=n+1}^{2n} \tfrac{q^{2k}}{k!(2n-k)!} \leq \tfrac{1}{(1-q^{2})^{2}} \sum_{k=0}^{\infty} \tfrac{q^{2k}}{k!(2n-k)!} = \tfrac{1}{(1-q^{2})^{2}} \tfrac{(1+q^2)^{2n}}{(2n)!}\,.
\end{align*}
Hence, by putting the two parts back together we obtain the following estimate 
\begin{align*}
\sum_{n = 1}^{\infty} |d_n\, t^{2n}| \leq 4 \, \tfrac{1 + q^4}{\left(1 - q^2\right)^2} \sum_{n = 1}^{\infty} \tfrac{1}{(2n)!} \left(\tfrac{ut (1+q^2)}{q}\right)^{2n} =4 \tfrac{1 + q^4}{\left(1 - q^2\right)^2} \big(\cosh(\vert w\vert\tfrac{1+q^2}{1-q^2}t)-1\big),\quad \forall t>0.
\end{align*}
So $\sum_{n = 1}^{\infty} d_n \,t^{2n}$ is indeed absolutely convergent for any fixed $t$.}
\end{rem}

\begin{rem}
About the existence of $\log$-terms:\\
{\rm  The presence in \eqref{heat1} of $\log$-terms is not a surprise:  in \cite[eq. (1.1)]{GilkeyGrubb}, the asymptotics of the heat trace
\begin{align}
\label{GG}
\Tr e^{-t\,P}\underset{t\to 0}{\sim} \sum_{l=0}^\infty a_l(P)\,t^{(l-m)/d} + \log t \,\sum_{k=1}^\infty b_k(P)\, t^k
\end{align}
is recalled for a classical positive elliptic pseudodifferential operator $P$ of order $d$ on a vector bundle on a $m$-dimensional compact smooth manifold. \\
The coefficients $b_k(P)=\underset{s=-k}{\Res} \,(s+k)\,\Gamma(s)\,\zeta_P(s)$ correspond to the poles or order one of $\zeta_P$ at the point $s\in -\N$. So, even if in \eqref{GG}, we could forget these terms since $\lim_{t\to 0} t^k\log t=0$, they pop up when one wants to get an exact heat trace formula and not only its asymptotics. Recall that here, $\zeta_\DD$ has poles of order two on negative integers.
}
\end{rem}

\begin{proof}[\it Proof of Theorem \ref{PropHeat}]
By the Proposition \ref{MellinLemma2} we have
\begin{align*}
\Tr e^{-t \,|\DD|} = \sum_{s_0 \in (\Sd_1 \cup - \N)} \,\underset{s = s_0}{\Res}\, \, t^{-s} \,\zD(s) \,\Gamma(s),
\end{align*}
and we obtain the following structure of the heat-trace using Mathematica \cite{Mathematica}:
\begin{align}
\Tr e^{-t\, |\DD|} & = \sum_{\alpha \in \Sd_1} \, \sum_{n = 0}^{2} a_{\a,n} \, (\log t)^n \, t^{-\alpha} + \sum_{n = 0}^{\infty} \wt{d}_n \,t^{n} \label{heat_coeff} \\
&\vc \tfrac{1}{\log^2 q} \Big[  \sum_{m \in \N} g_m\, (u t)^{2m} \, \log^2 t  + \sum_{m \in \N} \,\sum_{a \in \Z} h_{m,a} \,(ut)^{2m + \tfrac{2 \pi i}{\log q} a} \, \log t  \notag \\ 
& \hspace{5.25cm} + \sum_{m \in \N} \,\sum_{a \in \Z} c_{m,a} \,(ut)^{2m + \tfrac{2 \pi i}{\log q} a} \Big] + \sum_{n = 0}^{\infty} \wt{d}_n \, t^{n}.\label{heat_coef}
\end{align}

It should be stressed that the residues at real poles should be computed separately from the complex ones. The 
reason is that the function $\Gamma(s+n)$ appearing in the formula \eqref{zeta_D1} is singular at $s = -2m$, but regular at 
$s = -2m + \tfrac{2 \pi i}{\log q} \,a$ for any $a \in \Z^*$. Note that $\Gamma(s)^{-1}$, which cancels the third order 
pole in the $\zD$-function is suppressed by $\Gamma(s)$ from the inverse Mellin transform (\ref{invMellin}). In 
particular, this means that the residues at complex poles do not contribute to the $\log^2 t$ term in (\ref{heat1}). Moreover, it is convenient to consider separately the residues at $s \in -\N$ resulting from the singular part of the $\Gamma$-function. This is because our $\zD$-function \eqref{zeta_D1} is an infinite series and for $m \in \N$ we have
\begin{align*}
\underset{s=-2m}{\Res} \, t^{-s}\, \Gamma(s) \, \zD(s)
& = 4 \underset{s=-2m}{\Res} \, (ut)^{-s} \,\sum_{n=0}^{\infty} \, \tfrac{\Gamma (s+n)}{\Gamma(n+1)} \, \tfrac{q^{s+2n}}{(1-q^{s+2n})^2} \\
& = 4 \underset{s=-2m}{\Res} \, (ut)^{-s} \,\tfrac{\Gamma (s+m)}{\Gamma(m+1)} \, \tfrac{q^{s+2m}}{(1-q^{s+2m})^2} \\
& \hspace{2cm} + 4 (utq^{-1})^{2m} {\sum_{n=0,n\neq m}^{2m}} \tfrac{1}{\Gamma(n+1)} \, \tfrac{q^{2n}}{\left(1-q^{2(n-m)} \right)^2} \,\underset{s=-2m}{\Res} \Gamma (s+n), \\
\underset{s=-2m-1}{\Res} \, t^{-s} \,\Gamma(s)\, \zD(s)
& = 4 (utq^{-1})^{2m+1} \sum_{n=0}^{2m+1} \tfrac{1}{\Gamma(n+1)} \, \tfrac{q^{2n}}{\left(1-q^{2(n-m)-1} \right)^2} \,\underset{s=-2m-1}{\Res} \Gamma (s+n).
\end{align*}
Precise formulas for the coefficients are obtained with Mathematica
\begin{align}
&g_m  = \tfrac{(-1)^{m}2 }{(m!)^2}\,, \label{g}\\
&h_{m,0}  = \tfrac{ (-1)^{m} 4}{(m!)^2}\,\big(\log u-\psi ^{(0)}(m+1)\big)\,,\label{h}\\
&c_{m,0}  = \tfrac{(-1)^{m}}{3 (m!)^2} \big[6 \log ^2u - \log ^2 q+ 2 \pi ^2 - 12 \log (u)\, \psi ^{(0)}(m+1)\nonumber \\
&\hspace{2.4cm}+6 \psi ^{(0)}(m+1)^2-6 \psi ^{(1)}(m+1)\big]\,, \label{c}\\
&h_{m,a}  = -  \tfrac{4}{m!}\,\Gamma (-m-i\tfrac{2  \pi}{\log q}\,a )\,, \label{ha}\\
&c_{m,a}  = - \tfrac{4}{m!}\,\Gamma (-m-i\tfrac{2  \pi}{\log q}\,a)
\big( \log u - \psi ^{(0)}(-m -i\tfrac{2 \pi}{\log q}\,a) \big) \label{ca}\,, \\
&\wt{d}_n = 4 \left(\tfrac{u}{q}\right)^{n} {\sum_{k=0,k\neq n/2}^{n}} \tfrac{(-1)^{n-k}}{k!(n-k)!}\, \, q^{2k}(1-q^{2k-n})^{-2} \label{tDn}\,,
\end{align}
where $\psi^{(n)}$ stands for polygamma-functions of order $n$. \\
Using the definition of modified Bessel functions of 
first type (\ref{DJ}), the Euler's reflection formula and the fact that $\wt{d}_0 = \wt{d}_{2n+1} = 0$, 
one obtains the formula (\ref{heat1}) with $d_n\vc\wt{d}_{2n}$ for all $n\in \N$. 
Since we have already commented on the convergence of series over $a$ in (\ref{heat1}), the proof of Theorem \ref{PropHeat} is now complete. 
\end{proof}

Let us summarize the origin of each term in the heat kernel expansion \eqref{heat1}. The formula \eqref{heat1} is obtained by computations of the residues of the function $f_t: s \to t^{-s} \Gamma(s) \zD(s)$. This function has poles of third order at $s \in -2 \N$ and second order ones at $s \in -2 \N + \tfrac{2 \pi i}{\log q} \Z^*$. The Laurent expansion at $s = -2 m \in -2 \N$ of the function $f_t$ is:
\begin{align*}
f_t(s-2m) &= t^{-s+2m} \big( \gamma^{-1}_m \tfrac{1}{s-2m} + \gamma^{0}_m + \gamma^{1}_m (s-2m) + \OO((s-2m)^2) \big) \\
&\hspace{1.2 cm} \x \big( \zeta^{-2}_m \tfrac{1}{(s-2m)^{2}} + \zeta^{-1}_m \tfrac{1}{s-2m} + \zeta^{0}_m + \OO(s-2m) \big)
\end{align*}
with some coefficients $\gamma^i_j$ and $\zeta^i_j$. 
The following structure of coefficients (\ref{g}--\ref{tDn}) emerges:
\begin{itemize}
\item $g_m$ comes from $\zeta^{-2}_m \,\gamma^{-1}_m$;
\item $h_{m,0}$ comes from $\zeta^{-2}_m \,\gamma^{0}_m$ and $\zeta^{-1}_m\, \gamma^{-1}_m$;
\item $c_{m,0}$ comes from $\zeta^{-2}_m \,\gamma^{1}_m$ and $\zeta^{-1}_m \,\gamma^{0}_m$;
\item $\wt{d}_{n}$ comes from $\zeta^{0}_m \,\gamma^{-1}_m$.
\end{itemize}

Similarly, near $s = -2 m + \tfrac{2 \pi i}{\log q} \,a$ we have
\begin{align*}
& f_t(s-2m + \tfrac{2 \pi i}{\log q} a) = t^{-s+2m - \tfrac{2 \pi i}{\log q}\, a} \x \big( \gamma^{0}_{m,a} + \gamma^{1}_{m,a} (s-2m+ \tfrac{2 \pi i}{\log q}\, a) + \OO((s-2m+ \tfrac{2 \pi i}{\log q} \,a)^2) \big)\\
& \hspace{3.9cm} \x  \big( \zeta^{-2}_{m,a}\, \tfrac{1}{\left(s-2m+ \tfrac{2 \pi i}{\log q} \,a \right)^{2}} + \zeta^{-1}_{m,a}\, \tfrac{1}{s-2m+ \tfrac{2 \pi i}{\log q} a} + \zeta^{0}_{m,a} + \OO(s-2m+ \tfrac{2 \pi i}{\log q}\, a) \big) ,
\end{align*}
so
\begin{itemize}
\item $h_{m,a}$ comes from $\zeta^{-2}_{m,a}\, \gamma^{0}_{m,a}$;
\item $c_{m,a}$ comes from $\zeta^{-2}_{m,a}\, \gamma^{1}_{m,a}$ and $\zeta^{-1}_{m,a}\, \gamma^{0}_{m,a}$.
\end{itemize}

Let us note that for a compact Riemannian spin manifold of even dimension (like the sphere $S^2$ which has been deformed here) the zeta-function associated with the usual Dirac operator is regular for arguments in $\Sd_1$ defined in \eqref{S} \cite[Lemma 1.10.1]{Gilkey}, so only the coefficients $\wt{d}_{n}$ are ``classical''. 

\subsection{\texorpdfstring{About $t \to \infty$}{limit t infinite}}
\begin{lemma}
\label{t->infinity}
We have the following behaviour
\begin{align}
\label{t=infty}
\Tr \, e^{-t \,|\DD|}  \leq c \,e^{-|\omega|\,t} \text{ for } t\geq1.
\end{align}
\end{lemma}
\begin{proof}
Recall that if $A$ is a positive operator (possibly unbounded) with discrete spectrum $0=\lambda_0<\lambda_1\leq \lambda_2\leq \cdots$ such that $\Tr\,e^{-A} <\infty$, then $\Tr\,e^{-t\,A}=d + \mathcal{O}(e^{-t\lambda_1})$  for $t\geq1$ where $d \vc \dim\Ker(A)$: $\Tr\,e^{-t\,A}=d+\sum_{n=1}^\infty e^{-t\,\lambda_n}=d+e^{-t\,\lambda_1}\sum_{n=1}^\infty e^{-t\,(\lambda_n-\lambda_1)} \leq d+e^{-t\,\lambda_1}c \text{ for } t\geq1$, 
where $c\vc \sum_{n=1}^\infty e^{-(\lambda_n-\lambda_1)}=e^{\lambda_1}\Tr e^{-A}<\infty$. \\
Thus $\Tr \, e^{-t \,|\DD|}  \leq c \,e^{-|\omega|\,t} \text{ for } t\geq 1$, since $\lambda_1(|\DD|) = |\omega| [1]_q = |\omega|$.
\end{proof}

The reader might be suspicious about the large $t$ behaviour of the formula \eqref{heat1} since for instance 
$$
g(t)\,\log^2 t \underset{t\to \infty}{\sim} \tfrac{2}{\sqrt{u\pi}}  \,\cos(2ut-\tfrac{\pi}{4})\, \tfrac{\log^2 t}{\sqrt{t}}
$$
which does not seem to be compatible with \eqref{t=infty}. But, in fact the term $\sum_{n = 1}^{\infty} d_n \,t^{2n}$ does cancel exactly the logarithmic divergences, the constant terms and the oscillating part, giving the expected exponential decay. This sounds surprising, but seems to be a generic property of double exponential sums - see \cite[Example 12]{Flajolet}.

\subsection{\texorpdfstring{About $q \to 1$}{limit q=1}}
\label{q=1}

We are now interested in the behaviour of \eqref{heat1} when $q \to 1$, so denote by $\DD_q$ the operator introduced in \eqref{Dirac} and by $\DD_1$ the usual Dirac operator on the 2-sphere times $\vert w\vert$. 
\begin{prop}
\label{q->1}
We get the following limit:
$$\lim_{q\to 1}\Tr \,e^{-t\,\vert\DD_q\vert} = \Tr \, e^{-t\,\vert\DD_1\vert}  \text{ for any } t>0.$$
\end{prop}
\begin{proof}
Note that $\vert \DD_1\vert$  and $\vert \DD_q\vert$ have the same eigenspaces so they commute. The eigenvalues of $\vert\DD_q\vert- \vert \DD_1\vert$ are $\lambda_n(q)\vc \vert w\vert ([n]_q-n)$ for $n\in \N^*$. The following lemma implies that $\vert \DD_q \vert\downarrow \vert\DD_1 \vert$ as $q\uparrow 1$ and by the normality of the trace, $\Tr \, e^{-t\, \vert\DD_q\vert}  \uparrow \Tr \, e^{-t \, \vert\DD_1\vert} $.
\end{proof}

\begin{lemma}
For each $n\in \N^*$, $\lambda_n(q)$ is non-negative function of q which is strictly decreasing to $0$ when $0<q \uparrow 1$.
\end{lemma}

\begin{proof}
We can take $\vert w\vert=1$. With $x=-\log q\in [0,\infty[$, $$\lambda_n(q) = \tfrac{\sinh(n\log q)-n\sinh(\log q)}{\sinh(\log q)}=\tfrac{\sinh(nx)-n\sinh (x)}{\sinh(x)}$$ is positive since both numerator and denominator are positive. \\
Moreover, $\lambda_n(q) \underset{q\to 1}{\sim} \tfrac{n^3-n}{6}\,\log(q)^2>0$ for $n\in \N^*$.
\end{proof}
Note however, that the Theorem \ref{PropHeat} holds only for $0<q<1$, so one should not expect to have a well-defined limit of the RHS of \eqref{heat1}. The situation is somewhat similar to that encountered in the case of $SU_q(2)$ \cite{ILS}.

\subsection{Heat kernel for a simplified Dirac operator}

The eigenvalues of the Dirac operator for the Podle\'s sphere $\DD$ have exponential
growth (\ref{q_num}), slightly modified by a bounded term. For this reason, it is often
convenient to consider the operator $\DD_S$, which has  a similar form as
(\ref{Dirac}):
\begin{equation} \label{DS}
\DD_S \vc \left (\begin{smallmatrix} 0 & \bar{w} \,D_s \\  w \,D_s & 0 \end{smallmatrix}\right),
\qquad D_S:\, \ket{l,m}\in \H_\oh \to  \tfrac{q^{-(l+1/2)}}{q^{-1} - q}  \, \ket{l,m}\in \H_\oh.
\end{equation}
As diagonal operators, $\DD$ and $\DD_S$ commute and are related by
$$
\DD = \DD_S - \left( \tfrac{|w| q}{1-q^2} \right)^2 \DD_S^{-1}.
$$
The operator $\DD_S^{-1}$ is compact and trace-class, so it is obvious that taking $\DD_S$ instead of $\DD$ we preserve most of the properties of the spectral triple with the exception of the order-one condition.

\begin{remark}
\label{diff}
The heat-trace expansion of $|\DD_S|$ and $|\DD|$ are identical up to terms
regular at $t=0$ since 
$$0< \Tr e^{-t \,|\DD|} - \Tr e^{-t\, |\DD_S|} = 
\mathcal{O}(t \log^2 t), \; \text{ for } t>0. $$
\end{remark}

\begin{proof}
If $X$ is a bounded selfadjoint operator, then $ \lim_{t \to 0} \tfrac{1}{t} \norm{1 - e^{-tX}}  \leq \norm{X}$. This is due to 
$\norm{1 - e^{-tX}} = \norm{\sum_{n=1}^\infty (-1)^n \tfrac{t^n}{n!} X^n}\leq  \sum_{n=1}^\infty \tfrac{t^n}{n!} \norm{X}^n =  e^{t \norm{X}} -1.$
Now, we have
$$
0< \Tr e^{-t \,|\DD|} - \Tr e^{-t\, |\DD_S|}= \Tr \big(e^{-t \,|\DD|}\, (1 - e^{-t \,(|\DD_S| - |\DD|)})\big)  \leq  \norm{ 1 - e^{-t (|\DD_S| - |\DD|)}}\Tr e^{-t \,|\DD|}.
$$ 
Since $X=|\DD_s|-|\DD|$ is bounded, we know that above norm behaves like $\norm{|\DD_s|-|\DD|} t$ at $t \to 0^+$. Therefore the limit at $t \to 0^+$ of the difference of heat-traces vanishes because, using \eqref{heat1}, both $t \log t$ as well as $t \log^2 t$ vanish 
at $t=0$.
\end{proof}

Note that $\underset{q\to1}{\lim}\norm{e^{-t\,\vert\DD_S\vert}}=0$ for any $t>0$.
\\
The corresponding version of \eqref{heat1} for the heat-trace of simplified Dirac operator is
\begin{theorem}
\label{PropHeatDs}
For any $t>0$,
\begin{align*}
\Tr e^{-t\, |\DD_S|} = \tfrac{1}{\log^2 q} \,
\big[ 2 \log^2 (ut) +  h_S \big(\log (ut)\big)  \, \log (ut)  + c_S \big(\log (ut)\big)\big] + R(ut) \label{heat1s}
\end{align*}
where $u = \frac{|w|\,q}{1-q^2}$ and $h_S,\,c_S$ are the following periodic bounded $C^\infty$-functions on $\R$
\begin{align}
h_S(x) & \vc  4 \gamma
- 4 \sum_{a \in \Z^*} \Gamma (\tfrac{2\pi i}{\log q} \,a) \,e^{2\pi i a x}, \label{hS} \\
c_S(x)  & \vc  \tfrac{1}{3} ( \pi^2 + 6 \gamma^2 - \log^2 q) - 4
\sum_{a \in \Z^*} \Gamma (-\tfrac{2\pi i}{\log q} \,a)\,  \psi^{(0)} (\tfrac{2\pi i}{\log q}\, a)\, e^{2\pi i a x}, \label{cS}
\end{align}
while $R(ut)$ is  the regular part (with $\underset{t\to 0}{\lim}\, R(ut)=0$):
$$
R(x) = \sum_{m=1}^\infty \tfrac{(-1)^m \, q^{-m}}{(m)!(1-q^{-m})^2} \,\,x^{m}.
$$
\end{theorem}

\begin{proof}
As the proof goes directly along the lines of previous Theorem \ref{PropHeat}, we skip most of the arguments, which can be easily repeated. The meromorphic extension of the zeta-function associated with $\DD_S$ reads 
$$
\zeta_{\DD_S}(s) = \tfrac{4}{|\omega|} \, (1-q^2)^s\,(1-q^s)^{-2},
$$
which is precisely the $n=0$ term of \eqref{zeta_D1}. Thus, the residues related to the poles of $\zeta_{\DD_S}$ are just the coefficients of (\ref{g}--\ref{ca}) with  $m=0$:
\begin{align*}
g_0 & = 2, \\
h_{0,0} & = 4 \log u + 4\gamma, \\
c_{0,0} & = 2 \log^2 u + \tfrac{1}{3} \pi^2 - \tfrac{1}{3} \log^2 q + 4 \log u \gamma + 2\gamma^2, \\
h_{0,a} & = -4 \Gamma(-i \tfrac{2 \pi}{\log q} \,a), \\
c_{0,a} & = -4 \Gamma(-i \tfrac{2 \pi}{\log q}\, a) \left( \log u - \psi^{(0)}(-i \tfrac{2 \pi}{\log q}\, a) \right).
\end{align*}
This agrees with the formulae (\ref{hS}--\ref{cS}) and the coefficients (\ref{g}--\ref{ca}) and with Remark \ref{diff}, since the coefficients with $m>0$ contribute to the heat kernel with coefficients proportional to $t^m$.

The leading non-oscillatory parts of the heat kernel expansion of $|\DD_S|$ come from the residue at $s=0$ of 
$f_t(s) = t^{-s}  \Gamma(s)\zeta_{\DD_S}(s)$, whereas the oscillatory parts arise from the poles along the imaginary axis. Since the $\Gamma$-function is rapidly decreasing along the imaginary axis for both its real and imaginary parts, the resulting periodic function is a bounded periodic $C^\infty$-function on $\R$. 

Finally, the term $R(t)$ results from the residues of $f_t(s)$ at $s = -n \in - \N_+$, which can be obtained by taking only the $k=0$ part from the formula \eqref{tDn}.
\end{proof}

\section{Spectral action}
The spectral action for a spectral triple is defined as a functional on the space of all 
admissible Dirac operators:
\begin{align}\label{Sb}
S \vc \Tr f \left( \vert \DD\vert / \Lambda \right),
\end{align}
where $0<\Lambda$ is a cut-off and $f$ is a (positive) smooth cut-off function \cite{ConnesSA}. 
The formula (\ref{Sb}), although simple in its form, is difficult to calculate exactly due its non-local 
character. Recently, some tools have been developed \cite{ConnesNonPert,MPT1,MPT2} to perform 
non-perturbative computations of spectral action. Unfortunately, the fundamental assumption about the eigenvalues of the Dirac operator (polynomial growth) is not satisfied for the Podle\'s sphere, so we cannot use those methods. However, since we have at hand an exact formula for the heat kernel \eqref{heat1} we can use the (distributional) 
Laplace transform to compute the spectral action \textit{exactly}.

Let us also mention, that the usual asymptotic expansion of the spectral action (\ref{Sb}) for large $\Lambda$  is\cite{ConnesSA}
\begin{align}\label{Spert}
S \underset{\Lambda \to +\infty}{\sim} \sum_{k \in \Sd^+} f_k \,\Lambda^k \ncint |\DD|^{-k} + f(0) \,\zeta_{\DD}(0) 
+ \OO(\Lambda^{-1}),
\end{align}
and such behaviour cannot hold for the standard Podle\'s sphere, since the spectral triple neither has a simple dimension 
spectrum nor is regular.

\subsection{Spectral action computation}

To be able to use the results on the heat kernel of $\DD$, we have to put some restrictions on the regularizing function $f$. \\
Denote the Laplace transform of a measure $d\phi$ by 
\begin{align*}
\Lc(d\phi)(x)\vc \int_0^{\infty}  \,e^{-sx} \, d\phi(s), \quad \text{for } x \in \R^+.
\end{align*}
\begin{definition}
Let $\mathcal{C}$ be the set of functions $f= \Lc(d\phi)$ where $d\phi$ is a non-negative (not necessarily finite) Borel measure on $\R^+$, such that   
\begin{align}
\label{Phi}
 \int_0^{\infty} e^{s\,t} \,d\phi(s)< \infty,\text{ for any } t\geq 0.
\end{align}
\end{definition}

By the Hausdorff-Bernstein-Widder theorem \cite[p. 160]{Widder}, when $d\phi$ is finite, $f = \Lc(d\phi)$ always exists (independently of \eqref{Phi}) and is a completely monotonic function (see also \cite{ILV} for definition and properties). Let us stress that $d\phi$ is allowed to have support of Lebesgue measure 0, so the class $\mathcal{C}$ contains
\begin{align*}
f(x) = p(x) \, e^{-ax}, \text{ with } a>0 \text{ and } p(x)=\sum_{k=0}^d c_k\,x^k\text{ with } c_k\geq 0, \text{ so }d\phi = \sum_{k=0}^{d} c_k\, \delta_a^{(k)}.
\end{align*}
In particular, the heat operator case corresponding to $f(x) = e^{-t \,x}$ for $t>0$ is included in $\mathcal{C}$.
\\
All of  the functions $f$ resulting from a Laplace transform of positive Borel measures $d\phi$ with compact support in $]0,\infty[$ are encompassed by $\mathcal{C}$ since for these, the constraint \eqref{Phi} is clearly satisfied. As an example take for instance
\begin{align*}
f(x) = \tfrac{1}{x}(e^{-a x}-e^{-b x}), \text{ with } 0 < a < b, \text{ so } d\phi(s) = \Theta(s-a) - \Theta(s-b).
\end{align*}
In fact, any function $f$ in this subclass is of the form
\begin{align*}
f(x) = \tfrac{1}{x} \, \sum_n (e^{-a_n x} - e^{-b_n x}), \text{ with } 0 < a_0 < b_0 < a_1 < b_1 < \ldots,
\end{align*}
where the sum over $n$ might be infinite whenever it is convergent for any $x>0$.
\\
Moreover, $\mathcal{C}$ contains other functions like for instance
\begin{align*}
f(x) = e^{\frac{x^2}{4 a}}\left(1-\text{erf}(\tfrac{x}{2 \sqrt{a}}) \right), \text{ where erf is the error function}, \text{ so }d\phi(s) = \sqrt{ \tfrac{4 a}{\pi}} \, e^{-a s^2}\,ds.
\end{align*}

We now compute the spectral action, written with some energy scale $\Lambda$ even if the result is not asymptotic but exact.
\begin{theorem}
\label{spectralaction}
For $f \in \mathcal{C}$ and any $\Lambda>0$, the spectral action on standard Podle\'s sphere $\Sq$ reads
\begin{align}
\label{SA}
\Tr f \left( |\DD| / \Lambda \right) & = \sum_{\alpha \in \Sd_1}\, \sum_{p = 0}^{2} a_{\alpha,p} \,
\sum_{k = 0}^{p} (-1)^{p-k}  \tbinom{p}{k}  f_{\alpha,k} \,(\log \Lambda)^{p-k}\, \Lambda^{\alpha}   \\
& = \sum_{m \in \N} \, \sum_{n \in \Z} \, \sum_{p = 0}^{2} \notag \, a_{-2m+ \tfrac{2 \pi i}{\log q}\, n,p} \,
\sum_{k = 0}^{p} (-1)^{p-k}  \tbinom{p}{k}  \,f_{-2m+ \tfrac{2 \pi i}{\log q}\, n,k} \,(\log \Lambda)^{p-k} \,\Lambda^{-2m+ \tfrac{2 \pi i}{\log q}\, n} 
\end{align}
where $f_{\alpha,k} \vc \int_0^{\infty} s^{-\alpha} \log^k(s) \, d\phi(s)$ and $f = \Lc(\phi)$. 
The coefficients
\begin{align} \label{Scoef}
a_{\alpha,p} \vc p! \, \,\underset{s = \alpha}{\Res} \, \big( (s - \alpha)^p\, \Gamma(s)\, \zD(s) \big)
\end{align}
are related to formulas ({\rm \ref{g}--\ref{ca}}) by (recall that $ u = \tfrac{|w|\,q}{1 - q^2}$):
\begin{align}
&\left.
\begin{array}{ll}
a_{\alpha,2} = \tfrac{ u^{-\alpha}}{\log^2 q}  \,g_m \vspace{ 0.2cm}\\
a_{\alpha,1} = \tfrac{ u^{-\alpha}}{\log^2 q} \,h_{m,0}\hspace{0.3cm}\vspace{ 0.2cm} \\
a_{\alpha,0} =\tfrac{ u^{-\alpha}}{\log^2 q} \,c_{m,0} + d_m
\end{array}\right\}
 \text{ for } \alpha = -2m \in -2\N, \label{coef1}\\
& \left.
\begin{array}{ll}
a_{\alpha,2} = 0 \vspace{ 0.2cm}\\
a_{\alpha,1} = \tfrac{u^{-\alpha}}{\log^2 q} \, \,h_{m,-n}\vspace{ 0.2cm} \\
a_{\alpha,0} = \tfrac{u^{-\alpha}}{\log^2 q} \, c_{m,-n}
\end{array}\right\}
\text{ for } \alpha 
= -2m + \tfrac{2 \pi i}{\log q}\, n, \text{ with } m \in \N, n \in \Z^*. \label{coef2}
\end{align}
\end{theorem}

\begin{proof}
Since $f=\Lc(d\phi)$, we can write formally $f(t |\DD|) = \int_0^{\infty} e^{-st \,|\DD|}\,d\phi(s)$ for 
$t>0$. 
\\
By Theorem \ref{PropHeat} and Lemma \ref{t->infinity}, we know that there exist non-negative constants $c_i$ and $s_0$,  $s_1$ such that 
\begin{align}
 \label{c3}
0<\Tr e^{-s\vert \DD\vert} \leq 
\left\{ \begin{array}{ll} c_1 \log^2 s  &\text{ for } 0<s<s_0,\\
c_2\,e^{-c_3\,s}   &\text{ for } s>s_1.
\end{array} \right.
\end{align}
By the assumption \eqref{Phi} both $\int_0^{s_0} \log^2(s) \,d\phi(s)$ and $\int_{s_1}^\infty  e^{-c_3\,s} \, d\phi(s)$ are finite
 since $\log^2$ is integrable at the origin, so $f(\vert\DD\vert)$ is trace-class since 
$\Tr\big(f(\vert\DD\vert)\big) \leq \int_0^\infty \Tr (e^{-s\,\vert \DD\vert})\,d\phi(s) <\infty$. 
Moreover, the trace being normal, $\Tr \big( f(t |\DD|) \big) =  \int_0^{\infty} \Tr \,e^{-st \,|\DD|}\, \, d\phi(s)$.

Using \eqref{heat_coef}  with coefficients $a_{\a,p}$ given by the formulas (\ref{coef1}, \ref{coef2}), we obtain
\begin{align}
 \label{heatST}
\Tr e^{-st\, |\DD|} &= \sum_{\alpha \in \Sd_1} \,\sum_{p = 0}^{2} a_{\alpha,p} \log^p(s t)\, (st)^{-\alpha} \nonumber \\
 &= \sum_{\alpha \in \Sd_1}\, \sum_{p = 0}^{2} a_{\alpha,p}\, \sum_{k = 0}^{p}  
\tbinom{p}{k} \log^{p-k}( t)\, (st)^{-\alpha} \,\log^k(s).
\end{align}

Now, to obtain the formula \eqref{SA} we need to commute the integral over $s$ with the sum over $\alpha$. To this end we resort to Lebesgue dominated convergence theorem, which needs a uniform bound of the absolute value of the partial sums.

For the coefficients (\ref{g}--\ref{tDn}) we obtain the following estimates with $\eta = -\tfrac{2  \pi}{\log q}$\,:
\begin{align*}
|g_m|  & = \tfrac{2}{(m!)^2}\,,\\
|h_{m,0}| & = \tfrac{4}{(m!)^2}\, \left\vert \log u-\psi(m+1) \right\vert \leq \tfrac{4}{(m!)^2} \left( |\log(u)| +\vert\psi(1) \vert\, (m+1) \right),\\
|c_{m,0}| & = \tfrac{1}{3 (m!)^2} \big\vert 6 \log ^2u - \log ^2 q+ 2 \pi ^2 - 12 \log (u)\, \psi (m+1)\\
&\hspace{3.25cm} 
+6 \psi(m+1)^2-6 \psi ^{(1)}(m+1) \big\vert\\
& \leq \tfrac{1}{3 (m!)^2} \big( 6 |\log ^2u| + \log ^2 q+ 2 \pi ^2 + 12 |\log (u)| \, \vert\psi(1) \vert \,(m+1) \\
&\hspace{3.25cm} +6 \psi(1)^2\,(m+1)^2  + 6 \psi ^{(1)}(1) \big),\\
|h_{m,n}|  & = \tfrac{4}{m!}\,\left\vert \Gamma (-m-i\eta\,n ) \right\vert \leq \tfrac{4 \sqrt{2 \pi}}{m!}\, \left( m^2 + \eta^2 n^2 \right)^{-m/2-1/4} \, e^{-\pi \eta n/2} \, e^{1/(6 \sqrt{m^2+\eta^2 n^2})}\\
 & \leq \tfrac{4 \sqrt{2 \pi}}{m!}\, \left( m^2 + \eta^2 \right)^{-m/2-1/4} \, e^{-\pi \eta n/2} \, e^{1/6m},\\
|c_{m,n}| &  = \tfrac{4}{m!}\, \left\vert \Gamma (-m-i\eta\,n) 
\big( \log u - \psi(-m -i\eta\,n) \big) \right\vert  \\
& \leq \tfrac{4 \sqrt{2 \pi}}{m!}\, \left( m^2 + \eta^2 \right)^{-m/2-1/4} \, e^{-\pi \eta n/2} \, e^{1/6m} \left( |\log u| + \vert\psi(i \eta)\vert\,( m + 1 - \eta \vert n\vert) \right),\\
 |d_n| & \leq 4 \, \tfrac{1 + q^4}{\left(1 - q^2\right)^2} \tfrac{1}{(2n)!} \left(\tfrac{u (1+q^2)}{q}\right)^{2n}
\end{align*}
where we have used \cite[(2.1.19)]{Paris}, Remark \ref{rem3} and rough estimates of polygamma functions 
\begin{align*}
& 0\leq \vert\psi(x)\vert\leq (x+1) \,\vert\psi(1)\vert,\quad  0 <\psi^{(1)}(x+1) \leq \psi^{(1)}(1), \quad \text{ for } x\geq 0, \\
& \vert \psi(-x-i\eta y) \vert \leq \vert \psi(i \eta) \vert ( x + 1 + \eta \vert y\vert),  \qquad \text{ for } x \geq 0, y\in \Z^*
\end{align*}
based on large argument behaviour of the latter (essentially, $|\psi(z)|$ grows logarithmically with $|z|$ it is sufficient to bound it by its value at $0$ multiplied by $1+ |\Re(z)| + |\Im(z)|$).

Now, let us provide a uniform bound for the absolute value of the partial sums of \eqref{heatST}:
\begin{align*}
\left\vert \sum_{m=0}^{M}\, \sum_{n=-N}^{N} \, a_{-2m + i \eta n,p} \,(st)^{2m - i \eta n} \right\vert 
& \leq \sum_{m=0}^{M}\, \sum_{n=-N}^{N} \, |a_{-2m + i \eta n,p}|\, (st)^{2m} \\
& \leq \sum_{m=0}^{\infty} \,\sum_{n=-\infty}^{\infty} \, |a_{-2m + i \eta n,p}|\, (st)^{2m} \cv K_p(st).
\end{align*}

In view of formulas (\ref{coef1},\ref{coef2}) and the estimations on (\ref{g}--\ref{tDn}) obtained above, we conclude that the last power series has an infinite convergence radius and moreover, there exist constants $C_1, \, C_2 > 0$ such that for any $p \in \set{0,1,2}$, we have
\begin{align*}
K_p(st) \leq C_1 \, e^{C_2 \, t s}, \quad \text{ for any } t>0.
\end{align*}
Moreover, one can always refine the constants as to have $K_p(st)\, |\log^k s| \leq \wt{C}_1 \,e^{\wt{C}_2 t s}$ 
valid for any $t>0$ and any $k = 0,1,2$.
Thus, by  the assumption \eqref{Phi}, we have
\begin{align*}
\int_0^{\infty}K_p(st) |\log^k s | \, d\phi(s)\leq \wt{C}_1 \int_0^{\infty}  e^{\wt{C}_2 \,t s} \,d \phi(s) < \infty
\end{align*}
for any $t>0$ and any $k,p = 0,1,2$.

Hence, the conditions of Lebesgue dominated convergence theorem are met and we obtain
\begin{align*}
\Tr \big( f(t |\DD|) \big) & =  \int_0^{\infty}  \Tr \,e^{-st \,|\DD|} \,d\phi(s)  = \sum_{\alpha \in \Sd_1}\, \sum_{p = 0}^{2} a_{\alpha,p} \sum_{k = 0}^{p}  
\tbinom{p}{k} \log^{p-k}( t)\, t^{-\alpha} \int_0^{\infty}  s^{-\alpha}\,\log^k(s)\, d\phi(s)
\end{align*}
and the result follows by setting $t = \Lambda^{-1}$.
\end{proof}
Remark that in \eqref{SA}, $\alpha \in \C$ so $\Lambda^{\alpha}=\Lambda^{\Re (\alpha)} \Lambda^{i \, \Im (\alpha)} \in\C$ exactly like in Theorem \ref{PropHeat} where the oscillations are hidden in the definitions of $\wt{J}_{i \wt{a}}^{(n)}$.

In \eqref{SA}, the coefficients of the spectral action are computed in terms of the Laplace transform of the function $f$. We postpone to the Appendix, the computation of these coefficients directly in function of $f$.

\begin{rem}
\label{rem4}
Another presentation of the exact spectral action:\\
\rm{The formula \eqref{SA} is particularly useful for comparison with the standard asymptotic expansion \eqref{Spert}. On the other hand, if one does not insists on having an explicit power-like expansion in $\Lambda$, one could use directly the exact formula \eqref{heat1} instead of \eqref{heatST} to obtain
\begin{align*}
\Tr \big( f(|\DD| / \Lambda) \big) & =  \int_0^{\infty}  \Tr \,e^{-s \,|\DD| / \Lambda} \,d\phi(s) \\
 = \frac{1}{\log^2 q} & \, \Bigg[\log^2 \Lambda \,  \int_0^{\infty} g(s/\Lambda)   \,d \phi(s)\\
 &  \hspace{0.3cm}- \log \Lambda \, \int_0^{\infty}  \left( g(s/\Lambda) \log(s) + h_0(s/\Lambda) + \sum_{a \in \Z} h_a(s/\Lambda) \right) \,d \phi(s)\\
 & \hspace{0.3cm} + \int_0^{\infty} \Bigg( 2 g(s/\Lambda) \log^2(s) + h_0(s/\Lambda) \log(s) + \sum_{a \in \Z} h_a(s/\Lambda) \log(s)  \\
& \hspace*{3cm} +  c_0(s/\Lambda) + \sum_{a \in \Z} c_a(s/\Lambda) + \log^2 q \sum_{n=1}^{\infty} d_n s^{2n} \Lambda^{-2n} \Bigg) \Bigg]\,d \phi(s).
\end{align*}
Let us note, that we do not have to satisfy the demanding conditions of the Lebesgue dominated convergence theorem to derive the above formula. In fact, we only need the convergence of all of above integrals for any $\Lambda >0$, which includes $\int_0^{s_0} \log^2(s) \,d \phi(s)< \infty $ and $\int_{s_1}^\infty e^{-c_3\,s} \,d \phi(s) < \infty$ by \eqref{c3}. 
So, the above formula is satisfied for a significantly larger class of functions than $\mathcal{C}$, which include for example
\begin{align*}
f(x) = (x+a)^{-r}\text{ for any } & a>0,\, r>0, \text{ with }d\phi(s)=\Gamma(r)^{-1}\,s^{r-1}\,e^{-as}\,ds.
\end{align*}
since
\begin{align*}
\int_0^\infty e^{st}\,d \phi(s)& =(a-t)^{-r},\quad\text{for any }t<a,\\
\int_0^\infty s^{-\alpha} \,d \phi(s)& = \tfrac{(a)^{\alpha -r} \Gamma (r-\alpha )}{\Gamma (r)}\,,\\
\int_0^\infty s^{-\alpha} \,\log s\,d \phi(s) & = \tfrac{(a)^{\alpha -r} \Gamma\, (r-\alpha )}{\Gamma (r)} \left(\psi ^{(0)}(r-\alpha )-\log
   a\right), \\
\int_0^\infty s^{-\alpha} \,\log^2s\,d \phi(s)& =\tfrac{(a)^{\alpha -r} \Gamma (r-\alpha )}{\Gamma (r)} \,\left(\psi ^{(1)}(r-\alpha )+[\log a-\psi ^{(0)}(r-\alpha )]^2\right).
\end{align*}
}
\end{rem}

Let us now compare the computed spectral action with the one for the simplified Dirac operator \eqref{DS}. The notation is the same as in Theorem \ref{spectralaction}.

\begin{Prop}
\label{spectralsimplified}
For $f \in \mathcal{C}$, the spectral action on standard Podle\'s sphere $\Sq$ with the simplified Dirac operator \eqref{DS} reads
\begin{multline}\label{SAS}
\Tr f \left( |\DD_S| / \Lambda \right) \\
 = \sum_{n \in \Z} \, \sum_{p = 0}^{2} \, \wt{a}_{\tfrac{2 \pi i}{\log q}\, n,p} \,
\sum_{k = 0}^{p} (-1)^{p-k}  \tbinom{p}{k}  \,f_{\tfrac{2 \pi i}{\log q}\, n,k} \,(\log \Lambda)^{p-k} \,\Lambda^{\tfrac{2 \pi i}{\log q}\, n} + \sum_{n = 1}^{\infty} \, \wt{b}_n f_{-n,0} \Lambda^{-n},
\end{multline}
with
\begin{align} \label{coef1S}
\wt{a}_{0,2} = \tfrac{g_0}{\log^2 q}, \quad \wt{a}_{0,1} =  \tfrac{h_{0,0}}{\log^2 q}, \quad \wt{a}_{0,0} =  \tfrac{c_{0,0}}{\log^2 q},
\end{align}
\begin{align}
&\left.
\begin{array}{ll}
a_{\alpha,2} = 0 \vspace{ 0.2cm}\\
a_{\alpha,1} = \tfrac{u^{-\alpha}}{\log^2 q} \, \,h_{0,-n}\vspace{ 0.2cm} \\
a_{\alpha,0} = \tfrac{u^{-\alpha}}{\log^2 q} \, c_{0,-n}
\end{array}\right\}
\text{ for } \alpha 
= \tfrac{2 \pi i}{\log q}\, n, \text{ with } m \in \N, n \in \Z^* \label{coef2S}
\end{align}
and $\wt{b}_n = ( \tfrac{ |\omega| }{1-q^2}) ^{n} \tfrac{(-1)^n}{(n)!(1-q^{-n})^2}$\,.
\end{Prop}

It follows from Theorem \ref{spectralaction}, Proposition \ref{spectralsimplified} and Lemma \ref{diff} that:

\begin{Cor}
The spectral action for the simplified Dirac operator \eqref{DS} coincide with the one for 
the full Dirac operator \eqref{Dirac} up to terms, which vanish at $\Lambda \to \infty$;
that is, which are negligible for high energy scales.
\end{Cor}

\subsection{On the fluctuations of spectral action}

In this section we shall investigate how fluctuations of the Dirac operator,
$\DD \to \DA = \DD + \Ag$, affect the spectral action when $\Ag = \sum_i a_i [\DD,b_i]$, 
for $a_i, \, b_i \in \A$, is a selfadjoint one-form. The perturbation of the square of the Dirac
operator, $X_\Ag \vc \DA^2 - \DD^2 =\DD \Ag + \Ag \DD + \Ag^2$ is of order one ($X_\Ag \in \op^1$, compare Lemma \ref{lemma1as}). 
We always may assume that $\Vert X_\Ag\, \DD^{-2}\Vert <1$ by changing $\Ag$ into $\epsilon \Ag$ 
for a small enough $\epsilon >0$: indeed, observe first, using \eqref{DA} that 
\begin{align*}
X_\Ag \DD^{-2}=\DD\Ag\DD^{-2}+\Ag F|\DD|^{-1}+\Ag^2 \DD^{-2},\text{ and } \DD\Ag\DD^{-2}=\chi^{-1} F\Ag \vert \DD\vert^{-1}+ \mathcal{O}(\vert \DD\vert^0)\,\DD^{-2},
\end{align*}
so $X_\Ag \,\DD^{-2}$ is a sum of bounded operators. \\
By dilation, $\Vert X_{\epsilon \Ag} \,\DD^{-2} \Vert \leq \epsilon \Vert \DD \Ag \DD^{-2}\Vert + \epsilon \Vert \Ag F|\DD|^{-1} \Vert + \epsilon^2 \Vert \Ag^2 \DD^{-2} \Vert$, so for $\epsilon<\epsilon_\Ag$ where $\epsilon_\Ag$ is defined by $\epsilon_\Ag \Vert \DD \Ag \DD^{-2} + \Ag F|\DD|^{-1}\Vert + \epsilon_\Ag^2 \Vert \Ag^2 \DD^{-2} \Vert=1$, and we get
\begin{align}
\label{smallA}
\Vert X_{\epsilon \Ag} \,\DD^{-2} \Vert <1,\quad \forall \epsilon <\epsilon_\Ag\,.
\end{align}

\begin{theorem}
For any fluctuation $\DA$ of the Dirac operator $\DD$ by a small enough $\Ag$ (i.e. $\Vert X_\Ag \,\DD^{-2}\Vert < 1$), the 
heat kernel expansion of $\DA$ is the same as that of $\DD$ up to terms, which vanish at $t=0$. Consequently, the leading terms of the spectral action (that is, the terms, which do not vanish in $\Lambda \to \infty$ limit) do not depend on $\Ag$.
 \end{theorem}

\begin{proof}
We shall begin by investigating 
$|\DA|^{-s}$. Using $ x^{-s} = \tfrac{\sin (\pi s)}{\pi} \int_0^\infty \tfrac{\mu^{-s} }{x+\mu} \,d\mu$ for $0< \Re(s) <1$ and the functional calculus, we obtain:
\begin{align}
\label{sin}
|\DA|^{-s} = \tfrac{\sin (\pi s/2)}{\pi} \int_0^\infty \tfrac{\mu^{-s/2}}{\DA^2+\mu} \,d\mu, \quad \text{for } 0< \Re(s) <1.
\end{align}
The operator $X_\Ag$ of order one will be used in the power series expansion of $(\DA^2+\mu)^{-1}$.\\
Since $\Ag$ is small enough, 
\begin{align*}
(\DA^2+\mu)^{-1} &= 
\left( (1 +X_\Ag (\DD^2 +\mu)^{-1}) \,(\DD^2+\mu) \right)^{-1} =
(\DD^2+\mu)^{-1}  \sum_{n=0}^\infty \left( X_\Ag (\DD^2 +\mu)^{-1} \right)^n\\
& = (\DD^2+\mu)^{-1}+ (\DD^2+\mu)^{-1}\sum_{n=1}^\infty \left( X_\Ag (\DD^2 +\mu)^{-1} \right)^n
\end{align*}
and by \eqref{sin}
\begin{align}
\label{approx}
|\DA|^{-s} = |\DD|^{-s} +\tfrac{\sin (\pi s/2)}{\pi} \int_0^\infty d\mu \,\mu^{-s/2} \, (\DD^2+\mu)^{-1}\,Y,\quad Y\vc \sum_{n=1}^\infty \big( X_\Ag (\DD^2 +\mu)^{-1} \big)^n.
\end{align}
Since
\begin{align*}
\Vert Y\Vert & \leq \sum_{n=1}^\infty \Vert X_\Ag (\DD^2 +\mu)^{-1} \Vert^n \leq \sum_{n=1}^\infty \big(\Vert X_\Ag \DD^{-2}\Vert\, \Vert \DD^2(\DD^2+\mu)^{-1}\Vert\big)^n \leq \sum_{n=1}^\infty \Vert X_\Ag \DD^{-2} \Vert^n \leq C_\Ag,
\end{align*}
$Y$ is a bounded operator. Thanks to the relation $Y=X_\Ag(\DD^2+\mu)^{-1}\,(1+Y)$, we get 
\begin{align*}
\left\vert \Tr \big(\int_0^\infty d\mu \,\mu^{-s/2} \, (\DD^2+\mu)^{-1}Y\big) \right\vert &=\left\vert \Tr \big(\int_0^\infty d\mu \,\mu^{-s/2} \, (\DD^2+\mu)^{-1}X_\Ag (\DD^2+\mu)^{-1}(1+Y) \big) \right\vert \\
& \leq \Vert 1 + Y \Vert\int_0^\infty d\mu\, \mu^{-s/2}\, \Tr \big(\vert (\DD^2+\mu)^{-1}X_\Ag (\DD^2+\mu)^{-1} \vert \big).
\end{align*}
For a positive trace-class operator $U$ and a selfadjoint unbounded operator $V$, we have
$$
\Tr \vert UVU \vert = \Tr \vert U(V_+-V_-)U \vert \leq \Tr  UV_+U+\Tr UV_-U=\Tr U\vert V\vert U=\Tr \vert V \vert U^{2}.
$$
thus 
\begin{align*}
\Tr \big(\vert (\DD^2+\mu)^{-1}X_\Ag (\DD^2+\mu)^{-1} \vert \big)&\leq \Vert X_\Ag \vert \DD \vert^{-1}\Vert  \,\Tr \,\vert\DD\vert(\DD^2+\mu)^{-2} \\
&\leq \Vert X_\Ag \vert \DD \vert^{-1}\Vert \, \min \big( \Tr \vert \DD \vert^{-3},\,\tfrac{1}{4\mu} \Tr \vert \DD\vert^{-1}\big)
\end{align*}
using for the last inequality either $(\DD^{2}+\mu)^{-1} \leq \DD^{-2}$ or $\DD^{2}+\mu\geq 2 \sqrt{\mu}\, \vert \DD \vert$.\\
Let $\mu_0$ be defined by $\Tr \vert \DD \vert^{-3}=\tfrac{1}{4\mu_0}\,\Tr \vert \DD\vert^{-1}$. Then, gathering all inequalities,
\begin{align*}
&\hspace{-0.9cm}\left\vert \,\Tr\big(\tfrac{\sin (\pi s/2)}{\pi} \int_0^\infty d\mu \,\mu^{-s/2} \, (\DD^2+\mu)^{-1}Y\big)\,\right\vert\\
&\leq \Vert 1 + Y \Vert \, \Vert X_\Ag \vert \DD \vert^{-1}\Vert \, \left\vert\tfrac{\sin (\pi s/2)}{\pi} \right\vert \, \big[\Tr \vert \DD\vert^{-3}\,\int_0^{\mu_0} d\mu\, \mu^{-s/2}+\Tr \vert \DD \vert^{-1}\int_{\mu_0}^\infty d\mu\,\mu^{-s/2} \tfrac{1}{4\mu} \big]\\
&=\Vert 1 + Y \Vert\, \Vert X_\Ag \vert \DD \vert^{-1}\Vert \, \left\vert \tfrac{\sin (\pi s/2)}{\pi} \right\vert \, \big[\Tr (\vert \DD\vert^{-3})\,\tfrac{2\mu_0^{-s/2-1}}{2-s}+\Tr (\vert \DD \vert^{-1})\,\tfrac{\mu_0^{-s/2}}{2s} \big].
\end{align*}
Since the right hand side of last equality is regular function of $s$ for $0\leq \Re(s)<2$, the use of \eqref{approx} gives
\begin{align*}
\zeta_{\DD_\Ag}(s)=\zeta_\DD(s) + \text{ a regular function of $s$ for }0\leq \Re(s) < 2.
\end{align*}
Both $\zeta_{\DD_\Ag}(s)$ and $\zeta_{\DD}(s)$ are regular for $\Re(s)>0$, so if they difference is regular at $\Re(s)=0$,
the poles on the imaginary axis are identical for both functions. Therefore, in the heat trace the expansion (\ref{heat_coeff}) 
the terms that arise from $\Re(\alpha)=0$ coincide in both cases and the heat trace expansion is identical up to terms which
vanish at $t=0$. Similarly, in the spectral action expansion (\ref{SA}) the terms in the sum over $\alpha\in \text{Sd}_1=-2\N + i\tfrac{2 \pi}{\log q} \,\Z$ with $\Re(\alpha)=0$ are
the same, so the spectral action asymptotics is identical for $\DD$ and $\DD_\Ag$ up to terms, which vanish at 
$\Lambda \to \infty$.
\end{proof}

\section {Conclusions}

The study of the spectral triple over the standard Podle\'s sphere has brought a richness  
of surprising results. Having left the world of classical or almost classical Dirac operators, we
find new and interesting phenomena. Let us summarize the most important ones:
\begin{itemize}
\item The spectral zeta-function $\zD$ is not regular at $0$ and has poles at an infinite square 
lattice of points in the left half of the complex plane.
\item The dimension spectrum can be computed despite of the lack of regularity and contains second order poles.
\item The heat trace has a computable exact expansion, not only asymptotic for $t>0$.
\item The leading term of the heat-trace expansion is $\log^2 t$ (and, consequently, the large 
$\Lambda$-expansion of spectral action is $\log^2 \Lambda$).
\item The heat-trace is oscillatory in $t$ (and, consequently, the spectral action is oscillatory 
in $\Lambda$) due to complex poles in the dimension spectrum.
\item The leading parts of the spectral action do not depend on small perturbations 
of the Dirac operator. 
\end{itemize}

Each of the above listed phenomena is a new one, at least in the context of spectral triples 
and raises a lot of questions. First of all, the behavior of the zeta-function resembles that of
zeta functions for fractals. Although that similarity might prove to be superficial only, it is certainly 
worth further investigations. The appearance of $\log$ terms in the heat-trace and spectral action
suggests, on the other hand, the existence of singularities \cite{Lescure}. It would be interesting to 
find more common features of this type. Finally, as the heat-trace is an exact formula it would be
interesting to characterize the family of operators, for which the exact formula (of the type, which we found) is possible.

The coefficients of the heat-trace expansion are, in the case of classical Riemannian geometry, 
expressed in terms of geometrical objects. As we are dealing with a $0$-dimensional spectral
triple, no usual geometrical objects can be identified in that way. However, the stability of the leading terms with respect to small perturbations of the Dirac operator suggests that they also carry some information about the
object in question and its geometry, and it would be very interesting to identify them.
 
Moreover, the spectral action appears to have an absolutely undesired feature, as it includes oscillating 
terms. It is doubtful whether such behavior is admissible from the physical point of view, and what
possible consequences it would signify. 

Finally, let us mention that the $q \to 1$ limit cannot be easily recovered from the obtained heat-trace
formula. The heat kernel expansion diverges for $q \to 1$ and the reason is that the analytic continuation 
of $\zD$ simply does not hold for $q = 1$. This situation already appeared in $SU_q(2)$ case \cite{ILS}.

\section*{Appendix}

The class $\mathcal{C}$ of functions $f=\mathcal{L}(\phi)$ has a non empty intersection with the Schwartz space $\mathcal{S}(\R^+)$, so we could restrict to the case where $ f_{\alpha,k}=\< \phi, \Lc(\eta_{\a,k}) \>=\< \eta_{\alpha,k} ,f\> $,  for  $f\in \mathcal{S}(\R^+)\cap \,\mathcal{C}$ and some distributions $\eta_{\alpha,k} \in \mathcal{S}'(\R^+)$ to get the coefficient of \eqref{SA} directly in terms of $f$. \\
Thus here we compute the intrinsic distributions $\eta_{\alpha,k}$ .

Let us first fix notations: \\$F \ast G$ stands for convolution of distributions $F$ and $G$, 
$F^{\ast n} = \underset{n}{\underbrace{F \ast F \ast \ldots \ast F}}$ with the convention $F^{\ast 0} = \delta$.  
The Euler gamma constant is denoted by $\gamma$, $\Fp$ denotes the Hadamard finite part \cite{e},   $\Theta$ is the Heaviside step function and $H(k) = \sum_{m=0}^k \tfrac{1}{m}$ is the $k$-th harmonic number.

The following set of lemmas give the exact formulas for distributions $\eta_{\alpha,k}$ for any $\alpha \in \C$ and $k \in \N$.

\renewcommand{\theequation}{A.\arabic{equation}}
\setcounter{equation}{0} 
\renewcommand{\theappendixlemma}{A.\arabic{appendixlemma}}

\begin{appendixlemma}
\label{lm1} 
For $k\in \N$ and $\alpha \in \C \setminus -\N$, the inverse Laplace transform $\eta_{\alpha,k}$ of the function 
$s \in \R^+ \mapsto s^{-\alpha} \,\log^k s$ reads
\begin{align}\label{lm1_eq}
\eta_{\alpha,k}(t) = (-1)^k \tfrac{d^k\,}{d \alpha^k} \left( \tfrac{t^{\alpha-1}}{\Gamma(\alpha)} \right), \quad t>0.
\end{align}
\end{appendixlemma}

\begin{proof}
Let us first assume $\Re(\a)> 0$ and compute the Laplace transform:
\begin{align*}
\Lc \left( \tfrac{d^k\,}{d \alpha^k} \left( \tfrac{t^{\alpha-1}}{\Gamma(\alpha)} \right) \right)(s)
 = \int_0^{\infty} d t\, \tfrac{d^k}{d \alpha^k} \left( \tfrac{t^{\alpha-1}}{\Gamma(\alpha)} \right) e^{-st} 
= \tfrac{d^k}{d \alpha^k} \int_0^{\infty} d t \,\tfrac{t^{\alpha-1}}{\Gamma(\alpha)}\, e^{-st} 
 = \tfrac{d^k}{d \alpha^k} \,s^{-\alpha} = (-1)^k s^{-\alpha}  \log^k s.
\end{align*}
Now assume that $0 > \Re(\a) > -1$. Then we make use of the general formula, which holds for any distributional Laplace transform \cite{e}: $F(t) = \Lc^{-1}(f(s)) \; \Rightarrow t \, F(t) = \Lc^{-1}\left(-\tfrac{d}{ds} f(s) \right)$.\\
For $f(s) = s^{-\a} \log^k s$ we have
\begin{align*}
\Lc^{-1}\left(-\tfrac{d}{ds} f(s) \right) & = \a \, \Lc^{-1}\left( s^{-\alpha-1} \log^k s \right) - k \, \Lc^{-1}\left( s^{-\alpha-1} \log^{k-1} s \right)  \\
& = (-1)^k \a \, \tfrac{d^k\,}{d \alpha^k} \left( \tfrac{t^{\alpha}}{\Gamma(\alpha+1)} \right) + 
(-1)^{k} k \, \tfrac{d^{k-1}\,}{d \alpha^{k-1}} \left( \tfrac{t^{\alpha}}{\Gamma(\alpha+1)} \right)  \\
& = (-1)^k \tfrac{d^k\,}{d \alpha^k} \left(  \tfrac{\alpha\,t^{\alpha}}{\Gamma(\alpha+1)} \right) =
(-1)^k \tfrac{d^k\,}{d \alpha^k} \left( \tfrac{t^{\alpha}}{\Gamma(\alpha)} \right).
\end{align*}
Hence, by making use of the standard analytic continuation of the $\Gamma$ function, we have obtained  \eqref{lm1_eq} for $0 > \Re(\a) > -1$. Iterating this procedure, we can proceed to arbitrary low $\Re(\a)$, so the formula \eqref{lm1_eq} is well defined for arbitrary $\alpha \in \C \setminus -\N$.
\end{proof}
Let us remark, that since $\a$ is  complex number, the Laplace transform of $s \in \R^+ \mapsto s^{-\alpha} \,\log^k s$ is also a complex valued function.

 It is known that $\Fp \left( \tfrac{\Theta(t)}{t} \right) =\tfrac{d\,}{dt} \big(\Theta(t) \log t\big)$ \cite[(2.75)]{e} and in fact
 
\begin{appendixlemma}
\label{lm2}
The distributional inverse Laplace transform of $s\in \R^+ \mapsto \Theta(s)\,\log^k s $ where $k \in \N$ is given by 
\begin{align}\label{log}
\Lc^{-1}\big(\Theta(s)\,\log^k s  \big)(t) = (-1)^k \big[ (1-\delta_{k,0})\sum_{j=1}^k \tbinom{k}{j} \gamma^{j-1} \left[ \Fp \left( \tfrac{\Theta(t)}{t} \right) \right]^{\ast j}  + \gamma^k\, \delta(t) \big],\quad t>0.
\end{align}
\end{appendixlemma}

\begin{proof}
By definition, the Laplace transform of the distribution $F$ is $\Lc(F)(s) = \langle F(t), e^{-st} \rangle$, so 
$\Lc\big(\Fp ( \tfrac{\Theta(t)}{t} ) \big)(s)=\int_0^1\tfrac{e^{-st} -1}{t}\,dt +\int_1^\infty \tfrac{e^{-st}}{t}\,dt= -\gamma -\log s$. Thus    \eqref{log} is obtained for $k=1$ by $\Lc\big(\Fp ( \tfrac{\Theta(t)}{t} ) +\gamma\,\delta(t)\big)(s)=-\log s$ since $s>0$. 
Using $\Lc^{-1}(f g) = \Lc^{-1}(f) \ast \Lc^{-1}(g)$, we get $\Lc^{-1}(\Theta(s)\,\log^k s  )(t) = [ \Lc^{-1}(\Theta(s)\,\log s  ) ]^{\ast k} (t)$. 
\\
The formula (\ref{log}) follows from $\delta \ast f = f \ast \delta = f$ and the standard binomial expansion.
\end{proof}

\begin{appendixlemma}
\label{lm3}
We have the following formula for the distributional derivative of $\Fp \left( \tfrac{\Theta(t)}{t} \right)$:
\begin{align}\label{DnFp}
\frac{d^n\,}{d t^n} \Fp \left( \tfrac{\Theta(t)}{t} \right) & = (-1)^n n! \Fp \left( \tfrac{\Theta(t)}{t^{n+1}} \right) - H(n) \,\delta^{(n)}(t).
\end{align}
\end{appendixlemma}

\begin{proof}
 We should proceed by induction on $n$ by making use of the following formula \cite[(2.76)]{e}
\begin{align}\label{DFp}
\tfrac{d\,}{d t} \Fp  ( \tfrac{\Theta(t)}{t^k} ) = - k \Fp ( \tfrac{\Theta(t)}{t^{k+1}} ) + \tfrac{(-1)^k}{k!} \delta^{(k)}(t).
\end{align}
The first step for $n=1$ follows directly from (\ref{DFp}) for $k=1$. Now, assume $n>1$ and suppose (\ref{DnFp}) is satisfied for $n-1$. Then we have
\begin{align*}
\tfrac{d^n\,}{d t^n} \Fp ( \tfrac{\Theta(t)}{t} ) & = (-1)^{n-1} (n-1)! \,\tfrac{d\,}{d t} \Fp ( \tfrac{\Theta(t)}{t^{n}} ) - H(n-1)\, \delta^{(n-1)}(t)  \\
& = (-1)^n n! \Fp ( \tfrac{\Theta(t)}{t^{n+1}} ) - \tfrac{1}{n} \,\delta^{(n)}(t) - H(n-1) \,\delta^{(n)}(t)  \\
& = (-1)^n n! \Fp ( \tfrac{\Theta(t)}{t^{n+1}} ) - H(n) \,\delta^{(n)}(t).
\end{align*}
\end{proof}

\begin{appendixlemma}
\label{lm4}
The distributional inverse Laplace transform $\eta_{-n,k}$ of 
$s \in \R \mapsto \Theta(s)\,s^{n} \,\log^k s$ where $n \in \N$ and $0\neq k\in \N$ is given by 
\begin{align}\label{phiNK}
\eta_{-n,k}(t) & = (-1)^{k+n} n! \,\sum_{j = 1}^{k}\, \tbinom{k}{j} \,\gamma^{j-1} \sum_{i=1}^j (-1)^{i+1} \,H(n)^{i-1}\,
\Fp \left( \tfrac{\Theta(t)}{t^{n+1}} \right) \ast [ \Fp \left( \tfrac{\Theta(t)}{t} \right) ]^{\ast (j-i)} \notag \\
& \hspace{1cm}+ (-1)^k [ \gamma^{-1} \left( (1 - \gamma\, H(n) )^k - 1 \right) + \gamma^k ] \,\delta^{(n)}(t),\quad t>0.
\end{align}
\end{appendixlemma}

\begin{proof}
$\Lc^{-1} \big(\Theta(s)\,s^{n} \,\log^k s \big)(t) = \big( \Lc^{-1}(\Theta(s) \log^k s) \ast \Lc^{-1}(s^{n}) \big) (t)$, 
in which $\Lc^{-1}(s^{n})(t) = \delta^{(n)}$ and $\Lc^{-1}(\Theta(s) \, \log^k s )(t)$ is given by (\ref{log}). Since $\delta^{(n)} \ast T = D^n T$, where $D$ denotes the derivative operator and $D^n ( T^{\ast j} ) = ( D^n T) \ast T^{\ast(j-1)}$ we should proceed by induction on $j$. \\We claim that for $n\geq 1$, $j \geq 1$,
\begin{align}\label{DnFpj}
\tfrac{d^n\,}{d t^n} \left[ \Fp ( \tfrac{\Theta(t)}{t} ) \right]^{\ast j} &= (-1)^n n!\, \sum_{i=1}^j (-1)^{i+1} H(n)^{i-1} \Fp ( \tfrac{\Theta(t)}{t^{n+1}} ) \ast \left[ \Fp( \tfrac{\Theta(t)}{t} ) \right]^{\ast (j-i)} \notag \\
&\qquad + (-1)^{j} \,H(n)^j \,\delta^{(n)}(t).
\end{align}
The first step $j=1$ is just the formula (\ref{DnFp}). Now assume (\ref{DnFpj}) is satisfied for $j-1$, then with the help of the formula (\ref{DnFp}) we deduce
\begin{align*}
\tfrac{d^n\,}{d t^n} \left[ \Fp ( \tfrac{\Theta(t)}{t} ) \right]^{\ast (j)} & = \tfrac{d^n\,}{d t^n} \left[ \Fp ( \tfrac{\Theta(t)}{t} ) \right] \ast \left[ \Fp ( \tfrac{\Theta(t)}{t} ) \right]^{\ast(j-1)}  \\
& = (-1)^n n! \Fp ( \tfrac{\Theta(t)}{t^{n+1}} ) \ast \left[ \Fp ( \tfrac{\Theta(t)}{t}) \right]^{\ast (j-1)} - H(n) \tfrac{d^n\,}{d t^n} \left[ \Fp ( \tfrac{\Theta(t)}{t} ) \right]^{\ast (j-1)}  \\
& = (-1)^n n! \Fp ( \tfrac{\Theta(t)}{t^{n+1}} ) \ast \left[ \Fp ( \tfrac{\Theta(t)}{t} ) \right]^{\ast (j-1)}  \\
& \hspace{1cm} - H(n) (-1)^n n! \sum_{i=1}^{j-1} (-1)^{i+1} H(n)^{i-1} \Fp ( \tfrac{\Theta(t)}{t^{n+1}} ) \ast \left[ \Fp ( \tfrac{\Theta(t)}{t} ) \right]^{\ast (j-i-1)}  \\
& \hspace{1cm} - H(n) (-1)^{j-1}\, H(n)^{j-1}\, \delta^{(n)}(t)  \\
& = (-1)^n n! \sum_{i=1}^j (-1)^{i+1} H(n)^{i-1} \Fp ( \tfrac{\Theta(t)}{t^{n+1}} ) \ast \left[ \Fp ( \tfrac{\Theta(t)}{t} ) \right]^{\ast (j-i)}
(-1)^{j}\, H(n)^j\, \delta^{(n)}(t).
\end{align*}  
Now, the formula (\ref{phiNK}) is obtained by combining (\ref{log}) and (\ref{DnFpj}). In the $\delta^{(n)}(t)$ term of  (\ref{phiNK}) we have used the binomial expansion to compact the formula.
\end{proof}

\section*{Acknowledgments}

We thank Philippe Dumas for a fruitful correspondence and Artur Zając for his remarks on the draft. 
\\
Project operated within the Foundation for Polish Science IPP Programme ``Geometry and Topology in 
Physical Models'' co-financed by the EU European Regional Development Fund, Operational Program Innovative 
Economy 2007-2013.

\end{document}